\documentclass[11pt,reqno]{amsart}

\usepackage[a4paper,margin=1in]{geometry}
\usepackage{amsmath,amssymb,amsthm,amsfonts}
\usepackage{graphicx}
\usepackage{bm}              % Bold math symbols
\usepackage{booktabs}        % Better tables
\usepackage{cite}            % Compressed citations

\usepackage[latin1]{inputenc} % ou \usepackage[utf8]{inputenc}
\usepackage[T1]{fontenc} % ou \usepackage[OT1]{fontenc}
\usepackage{hyperref}
\usepackage{stmaryrd}
\usepackage{mathtools}
\usepackage{amsmath} 
\usepackage{mathrsfs} 
\usepackage{amssymb}
\usepackage{amsthm}
\usepackage{enumitem}
\usepackage{diagbox}
\usepackage{bm}
\usepackage{xcolor}
\usepackage[capitalise,nameinlink]{cleveref}
\usepackage{tabto}
\usepackage[usestackEOL]{stackengine}[2013-10-15]
\usepackage{subcaption}

\newtheorem{thm}{Theorem}[section]

\newtheorem{prop}{Proposition}[section]
\newtheorem{lem}{Lemma}[section] % separate counter
\newtheorem{defin}{Definition}[section]

\newcommand{\revA}[1]{{\color{black}#1}}
\newcommand{\revB}[1]{{\color{black}#1}}
\newcommand{\revC}[1]{{\color{black}#1}}
\theoremstyle{plain}
\newtheorem{theorem}{Theorem}[section]

\theoremstyle{definition}

\theoremstyle{remark}

\title[]{Error Estimates for Sparse Tensor Products of B-spline
Approximation Spaces}

\author{Cl\'ement Guillet$^\dagger$}

\date{\today \\ $^\dagger$Concace, Inria Center at the University of Bordeaux, Talence, France. \\ e-mail: \texttt{clement.guillet@inria.fr}}

\begin{document}

\begin{abstract}
This work introduces and analyzes B-spline approximation spaces defined on general geometric domains obtained through a mapping from a parameter domain.
These spaces are constructed as sparse-grid tensor products of univariate spaces in the parameter domain and are mapped to the physical domain via a geometric parametrization.
Both the univariate approximation spaces and the geometric mapping are built using maximally smooth B-splines. We construct two such spaces, employing either the sparse-grid combination technique or the hierarchical subspace decomposition of sparse-grid tensor products, and we prove their mathematical equivalence. Furthermore, we derive approximation error estimates and inverse inequalities that highlight the advantages of sparse-grid tensor products. Specifically, under suitable regularity assumptions on the solution, these spaces achieve the same approximation order as standard tensor product spaces while using significantly fewer degrees of freedom.
Additionally, our estimates indicate that, in the case of non-tensor-product domains, stronger regularity assumptions on the solution, particularly concerning isotropic (non-mixed) derivatives, are required to achieve optimal convergence rates compared to sparse-grid methods defined on tensor-product domains.
\end{abstract}

\maketitle

% ----------------------------------------------------
% Keywords and MSC
% ----------------------------------------------------

% ----------------------------------------------------
% 1. Introduction
% ----------------------------------------------------

\section{Introduction}
Many problems in computational science involve the approximation of high-dimensional functions. However, representing such functions on isotropic, uniformly refined meshes quickly becomes computationally intractable as the dimensionality increases beyond three or four. This phenomenon, known as the curse of dimensionality~\cite{bellman61}, refers to the exponential growth of computational cost with respect to the number of dimensions. Sparse tensor-product constructions, commonly referred to as sparse-grid (SG) methods, were introduced primarily to alleviate this curse by reducing the computational complexity while maintaining high approximation accuracy.

Sparse-grid methods~\cite{bungartz04,garcke13} were originally developed for the interpolation of high-dimensional functions and have since been extended to the approximation of partial differential equations~\cite{griebel92,griebel98,griebel07}. Under the assumption that the solution has bounded mixed derivatives, an accurate approximation, achieving nearly the same order of accuracy as standard tensor-product methods using isotropic uniform meshes can be computed at a reduced cost. 
Convergence rates of sparse-grid approximations using piecewise linear functions are well established~\cite{bungartz04,pflaum99}. Specifically, the $L^2$ error of sparse-grid approximations scales as $\mathcal{O}(h^2|\log h|^{d-1})$, where $h$ denotes the mesh size and $d$ the dimension. Constructing the sparse-grid approximation requires only $\mathcal{O}(h^{-1}|\log h|^{d-1})$ basis functions whereas standard tensor-product approximations require $\mathcal{O}(h^{-d})$ basis functions to achieve an $L^2$ error of $\mathcal{O}(h^2)$. These results were extended to polynomials~\cite{barthelmann00} and piecewise Lagrange polynomials~\cite{Achatz:2003aa, bungartz98a,bungartz98,bungartz04} with an $L^2$ error scaling as $\mathcal{O}(h^{p+1}|\log h|^{d-1})$, for polynomials of degree $p$. Similar results have also been proven for tensor-product operators~\cite{griebel12,griebel13a} and for non-tensor-product operators~\cite{griebel13}.

Sparse-grid methods were initially restricted to generic tensor-product domains. Some early works proposed to use of transfinite interpolation to generate curvilinear domains~\cite{bungartz98,Dornseifer:1996aa}. However, the regularity assumptions on the solution may no longer hold if the geometric mapping is insufficiently smooth, as is the case with transfinite interpolation. More recently, to overcome this limitation, sparse grids were introduced in the context of isogeometric analysis (IGA)~\cite{BECK2018128}, which can be viewed as a natural extension of finite element analysis using high-order, highly regular B-spline or non-uniform rational B-spline (NURBS) basis functions.
The main idea of IGA, initially introduced in~\cite{HUGHES20054135} is to employ the same type of functions, that is B-splines or NURBS, to define both the computational domain and the unknown solution of the differential equations. This approach allows for more accurate representation of complex geometries and enables exact descriptions of common geometries such as circles or ellipses. In~\cite{BECK2018128}, the authors introduce a sparse-grid isogeometric solver based on the so-called sparse-grid combination technique\cite{griebel90}, which consists of forming a linear combination of several inexpensive tensor-product approximations to recover an approximation of nearly the same order as that of full tensor-product discretizations. Numerical experiments demonstrated that the expected convergence rate of $\mathcal{O}(h^{p+1}|\log h|^{d-1})$ is achieved for sparse-grid approximations using NURBS basis functions of degree $p$. 

However, rigorous approximation error estimates for sparse-grid spaces on general, non-tensor-product domains are still lacking. Understanding such approximation properties is of fundamental importance in the context of Galerkin methods applied to the variational formulations of differential equations. Indeed, error estimates in appropriate Sobolev norms guarantee the convergence of the methods through the application of C\'ea's or Strang's lemma. The first error estimates for IGA approximation spaces were derived in the case of $h$-refinement with fixed polynomial degree~\cite{bazilev11}. The case of anisotropic meshes, involving different mesh sizes in each coordinate direction, was addressed in~\cite{veiga12}. This analysis was later extended to include the explicit dependence on the polynomial degree~$p$~\cite{veiga11}, although it did not cover the most relevant case of maximally smooth B-splines. This gap was filled in~\cite{takacs16}, where error estimates and inverse inequalities with constants independent of the spline degree and the mesh size were established for maximally smooth B-splines. \revB{An extension of this result to arbitrary grids with non-uniform spacing was derived in~\cite{sande19}, with a sharper, near-optimal constant. The result was further extended in~\cite{sande19-1,sande22} to splines of arbitrary smoothness defined on isogeometric multi-patch domains.}
 
In this work, we introduce and analyze approximation spaces based on sparse-grid tensor products of B-splines. 
This extends the work of~\cite{takacs16,sande19,sande19-1}, where such estimates were established for univariate B-splines of maximal smoothness and for $d$-dimensional spaces constructed from full tensor products. Here, we generalize these results to sparse-grid tensor-product spaces in the parameter domain and their images in the physical domain via a geometric B-spline mapping. For simplicity, we restrict our analysis to B-spline functions; however, the extension to NURBS functions, as in~\cite{bazilev11}, is straightforward.
Following~\cite{takacs16}, we restrict our analysis to single-patch geometric domains; however, an extension to multi-patch domains, along the lines of~\cite{takacs18,sande19-1}, could also be considered.

We propose two distinct approaches to construct such spaces. The first approach is based on the sparse-grid combination technique, while the second relies on the hierarchical subspace decomposition of sparse-grid tensor products. We show that, despite their different constructions, the two approximation spaces are identical and can thus be considered a single, unified space.  Here, we exclusively focus on the approximation properties of the space, leaving implementation and performance considerations to future work. 

%The objectives of this paper are twofold. 
%First, we provide a theoretical foundation, with approximation error estimates, for using sparse-grid approximation spaces with maximally smooth B-spline functions in Galerkin methods, as employed in IGA framework. 
%As experienced in~\cite{BECK2018128} and demonstrated in this paper, sparse-grid techniques reduce the computational cost of IGA solvers by decreasing the size of the approximation space and the number of degrees of freedom, while maintaining nearly the same approximation properties. 
%Our analysis further shows that the approximation space can be equivalently constructed using either the combination technique or a hierarchical subspace decompositionas both constructions yield the same space. Second, this work establishes the theoretical groundwork for employing maximally smooth B-splines in IGA for a broader class of applications. 

%For instance, finite-element applications using sparse-grid methods, such as Sparse-Particle-In-Cell methods in plasma physics~\cite{deluzet_sparse_2022,garrigues_2024_1,garrigues_2024_2,GUILLET2025113739}, are typically restricted to rectangular geometries. 
%These applications could benefit from the flexibility of IGA discretizations, which enable simulations in more complex spline-based domains. 
%Moreover, such applications may also take advantage of the improved accuracy observed for regular solutions when using maximally smooth B-splines~\cite{EVANS20091726,HUGHES20084104}.

\subsubsection{The Main Results}
We summarize in the following the two main results of this paper, \cref{thm:1} and \cref{thm:3}. 
These results provide \emph{a priori} error estimates and inverse inequalities in Sobolev norms for the sparse-grid approximation spaces, with constants independent of the mesh size $h$, the spline degree $p$, and the B-spline regularity. 

We consider the physical domain constructed by the application of a geometric mapping $\bm{F}$ to a tensor-product parameter domain $\hat{\Omega}:=(0,1)^d$, i.e.,
 \begin{align}
 \Omega := \bm{F}(\hat{\Omega}).
 \end{align}
We denote by $S_{p,h}^{(1)}(\Omega)$ the sparse-grid approximation space of B-splines with mesh size $h=2^{-n}$, where $n\in\mathbb{N}_0$, degree $p\in\mathbb{N}_0$ and maximum regularity $\mathcal{C}^{p-1}(\Omega)$.
\begin{theorem}%[Approximation error estimate]
\label{thm:1} 
Let $u \in H^{dq}(\Omega)$ with $q \in \mathbb{N}_0$, then for all $r \in \mathbb{N}_0$ satisfying $0 \leq r \leq q \leq p+1$, there exists a spline approximation $u_{p,h}^{(1)}\in S_{p,h}^{(1)}(\Omega)$ such that
\[
|u - u_{p,h}^{(1)}|_{H^{r}(\Omega)} 
\lesssim  h^{\,q-r} |\log h|^{\,d-1} \|u\|_{H^{dq}(\Omega)}.
\]
\end{theorem}
\begin{theorem}
\label{thm:3} 
For any $q\in \mathbb{N}_0$ satisfying $0\leq q \leq p$, the inverse inequality:
\[
|u_{p,h}^{(1)}|_{H^{q}(\Omega)} \lesssim  h^{-q} |\log h|^{d/2} \|u_{p,h}^{(1)}\|_{L^2(\Omega)}
\]
holds true for all spline approximations $u_{p,h}^{(1)}\in \tilde{S}_{p,h}^{q,(1)}(\Omega)$, where $ \tilde{S}_{p,h}^{q,(1)}(\Omega)\subset S_{p,h}^{(1)}(\Omega)$ is a spline subspace defined later in~\cref{sec:inv}.
\end{theorem}

The symbol $\lesssim$ refers to an inequality up to a constant that does not depend on the mesh size $h$ or the spline degree $p$, but may depend on the shape and the size of the domain $\Omega$. The notation $\log := \log_2$ is used throughout the manuscript.

This paper is organized as follows. In \cref{sec:1}, we introduce the main notation related to Sobolev spaces and the associated norms. In \cref{sec:2}, we describe the construction of regular approximation spaces based on univariate B-splines and their extension to anisotropic $d$-dimensional splines via tensor products. In \cref{sec:3}, we present the sparse-grid approximation spaces in tensor-product domains. In~\cref{sec:iga}, we extend the results to more general geometric domains \revA{and discuss the complexity and computational properties of the approximation spaces}. \revC{All proofs are provided in the appendix, in~\cref{apd:1}, to improve the readability of the paper.}

%---------------------------------------------------------------
\section{Preliminaries}
In this preliminary section, we introduce the main notations for spaces and norms that will be used in our analysis.
\label{sec:1}
\subsection{Sobolev spaces and norms}

Let $d\in \mathbb{N}$ denote the dimension of the physical domain of interest, which we take as a generic open set $\Omega \subset \mathbb{R}^d$. We denote by $W^{q,p}(\Omega)$ the standard $L^p$ Sobolev space of order $q \in \mathbb{N}_0$ with $p \in [1, +\infty]$, and by $H^q(\Omega)$ when $p=2$. Specifically, these spaces are defined as
\begin{align*}
&W^{q,p}(\Omega) := \left\{u: \Omega \rightarrow \mathbb{R} \;\middle|\; D^{\bm{\alpha}}u \in L^p(\Omega), ~~ |\bm{\alpha}|_1\leq q \right\}, \\
&H^q(\Omega) := W^{q,2}(\Omega).
\end{align*}
In the above definitions, the derivative operator \revB{is defined, for multi-indices $\bm{\alpha} = (\alpha_1,\ldots,\alpha_d) \in \mathbb{N}^d$, by
\[
D^{\bm{\alpha}}u := \frac{\partial^{\alpha_1}}{\partial x_1^{\alpha_1}}\cdots \frac{\partial^{\alpha_d}}{\partial x_d^{\alpha_d}}u,
\]}
For a smooth vector function $\bm{F} = (F_1, \ldots, F_d) \in (\mathcal{C}^\infty(\Omega))^d$ \revB{and multi-indices $\bm{\alpha} = (\alpha_1,\ldots,\alpha_d) \in \mathbb{N}^d$, 
we introduce the multi-derivative tensor of $\bm{F}$, defined componentwise by
\[
(\bm{\nabla}^{\bm{\alpha}} \bm{F})_{i} :=D^{\bm{\alpha}} F_i,
\quad i = 1,\ldots,d.
\]}
%In this work, we focus on the $L^2$-based Sobolev space $H^q(\Omega)$. 
We define the inner product on $H^q(\Omega)$ for $u,v \in H^q(\Omega)$ by
\[
(u,v)_{H^q(\Omega)} 
:= \sum_{|\bm{\alpha}|_1 \le q} \int_{\Omega} D^{\bm{\alpha}} u(\bm{x}) \, D^{\bm{\alpha}} v(\bm{x}) \, d\bm{x}, \quad \text{where}~ |\bm{\alpha}|_1 := \sum_{i=1}^d |\alpha_i|,
\]
and the associated seminorm and norm, respectively, by
\begin{align*}
|u|^2_{H^q(\Omega)} 
:= \sum_{|\bm{\alpha}|_1 = q} \| D^{\bm{\alpha}} u \|_{L^2(\Omega)}^2, \quad
\|u\|^2_{H^q(\Omega)} 
:= (u,u)_{H^q(\Omega)}.
\end{align*}
%Sparse-grid methods require stronger assumptions on the mixed regularity of the solutions than standard tensor-product methods. To this end, we consider function spaces suitable for their analysis. 
We define the $L^2$-based Sobolev space with dominating mixed smoothness of order $q \in \mathbb{N}_0$, also called the mixed Sobolev space, by
\[
H_{\mathrm{mix}}^q(\Omega) := \left\{u: \Omega \rightarrow \mathbb{R} \;\middle|\; D^{\bm{\alpha}}u \in L^2(\Omega), ~~ |\bm{\alpha}|_\infty\leq q \right\}, \quad \text{where}~ |\bm{\alpha}|_{\infty}:=\max_{i=1,\ldots,d}  |\alpha_i|.
\]
The associated seminorm and norm are defined by
\begin{align*}
|u|^2_{H_{\mathrm{mix}}^q(\Omega)} := \sum_{|\bm{\alpha}|_\infty=q} \| D^{\bm{\alpha}} u \|_{L^2(\Omega)}^2, \quad 
\|u\|^2_{H_{\mathrm{mix}}^q(\Omega)} 
:=\sum_{m=0}^{q} |u|_{H_{\mathrm{mix}}^m(\Omega)}^2.
\end{align*}
The standard Sobolev space is a subspace of the mixed Sobolev space, and the following relation holds between the full norms:
\begin{align}
\label{eq:13}
\|u\|_{H^{q}(\Omega)} \le \|u\|_{H_{\mathrm{mix}}^q(\Omega)} \le \|u\|_{H^{dq}(\Omega)}.
\end{align}

%-----------------------------------------------------------------------------

\section{Regular tensor-product approximation spaces}
\label{sec:2}
In this section, we introduce finite-dimensional approximation spaces
\[
S_{p,h_{\bm{\ell}}}(\hat{\Omega}) \subset H^q(\hat{\Omega})
\]
defined on the tensor-product domain $\hat{\Omega}=(0,1)^d$. These spaces provide increasingly accurate approximations in \(H^q(\hat{\Omega})\) as \(h_{\bm{\ell}} \to 0\) and will later serve as building blocks for the construction of sparse-grid approximation spaces. 

\subsection{Parameter space meshes}
We consider a family of meshes $\{\hat{\Omega}_h\}_h$ on a tensor-product parameter domain $\hat{\Omega}:=(0,1)^d$, where $h$ denotes the family index and represents the global mesh size. 
In this work, we \revB{restrict to uniform grid spacing and} consider a dyadic subdivision of the domain. Let $\bm{\ell}=(\ell_1,\ldots, \ell_d) \in \mathbb{N}^d$ be a multi-index specifying the mesh level in each coordinate direction. The corresponding mesh sizes are given by the $d$-dimensional vector
\[
h_{\bm{\ell}} := (2^{-\ell_1}, \ldots, 2^{-\ell_d}) \in \mathbb{R}^d. 
\]
These mesh sizes partition the $d$-dimensional parameter domain into elements constructed by tensor products of unidimensional dyadic meshes:
\[
\hat{\Omega}_{h_{\bm{\ell}}} 
:= \bigotimes_{i=1}^d \hat{\Omega}_{h_{\ell_i}}, 
\quad \text{where} \quad
\hat{\Omega}_{h_{\ell_i}} := \bigcup_{j=0}^{h_{\ell_i}^{-1}-1} \bigl( h_{\ell_i} j, \, h_{\ell_i} (j+1) \bigr).
\]
The family $\{\hat{\Omega}_{h_{\bm{\ell}}}\}_{\bm{\ell}}$ consists of anisotropic meshes indexed by the mesh levels $\bm{\ell}$, where we allow the mesh size to vary along different directions. This flexibility is essential for defining sparse-grid approximations.

%-----------------------------------------------------------------------------
\subsection{Univariate B-splines spaces}
In this section, we consider the unidimensional case corresponding to $d=1$, i.e., $\hat{\Omega}=(0,1)$. The mesh level is denoted by $\ell\in \mathbb{N}$.  Let $\mathbb{P}^p$ be the space of polynomials of degree $p \in \mathbb{N}_0$ and \revB{$\mathcal{C}^k(\hat{\Omega})$ the space of functions with continuous derivatives of
order $0,1, \ldots,k$ on $\hat{\Omega}$, with $k\in \mathbb{N}_0$.}
\begin{defin}
The space of univariate B-splines is defined by
\[
S_{p,k,h_\ell}(\hat{\Omega}) := \left\{ u \in \revB{\mathcal{C}^k(\hat{\Omega})} \;\middle|\; 
u|_{(h_\ell j,\, h_\ell (j+1))} \in \mathbb{P}^p, \ \forall j = 0,\ldots,h_\ell^{-1}-1 \right\},
\]
where the parameter $k$, with $0 \leq k \leq \revB{p-1}$, specifies the spline regularity.  
\end{defin}

We restrict ourselves to the case of B-splines with maximal smoothness, i.e., \revB{$k = p-1$}, and denote
\[
S_{p,h_\ell}(\hat{\Omega}) := S_{p,p-1,h_\ell}(\hat{\Omega}).
\]

For $r\in\mathbb{N}_0$ verifying $0\leq r\leq p+1$, we introduce the seminorm $H^r$-orthogonal projection operator onto the spline space and its complementary operator:
\begin{align*}
&\pi_{p,h_\ell} : H^r(\hat{\Omega}) \rightarrow  S_{p,h_\ell}(\hat{\Omega}), \\
& (\operatorname{I}-\pi_{p,h_\ell}) : H^r(\hat{\Omega}) \rightarrow  H^r(\hat{\Omega})/S_{p,h_\ell}(\hat{\Omega}),
\end{align*}
defined such that
\[
|(\operatorname{I}-\pi_{p,h_\ell})\, u |_{H^r(\hat{\Omega})} = \min_{v \in S_{p,h_\ell}(\hat{\Omega})} |u - v |_{H^r(\hat{\Omega})}.
\]

The main result of this section, originally presented in~\cite[Theorem 7.3]{takacs16}, \revB{further extended in~\cite[Theorem 3.1]{sande19} to arbitrary grids and improved with a sharper constant}, provides an estimate of the approximation error in the $H^r$ seminorm for B-spline spaces up to maximum regularity, with a constant independent of both the spline degree and the mesh size. 
We recall this result here, expressed with our notations. 
\begin{lem}%[Univariate approximation error]
\label{lem:2}
Let $u \in H^q(\hat{\Omega})$ with $q \in \mathbb{N}_0$. For all mesh level $\ell \in \mathbb{N}$, each spline degree $p \in \mathbb{N}_0$, and any $r \in \mathbb{N}_0$ satisfying $0 \leq r \leq q \leq p+1$, the following estimate holds true:
\[
|(\operatorname{I} - \pi_{p,h_\ell}) \,u|_{H^r(\hat{\Omega})} \leq C_1(q,r)\, h_\ell^{q-r} \, |u|_{H^q(\hat{\Omega})}, \quad \text{with}~~ C_1(q,r)= \revB{\left(\frac{1}{\pi}\right)^{q-r}}.
\]
\end{lem}
\revB{A similar result was initially introduced in~\cite[Theorem 7.3]{takacs16} with a constant independent of the spline degree, given by
\[
C_1^\star(q,r) = (\sqrt{2})^{q-r},
\] 
and under the assumption of sufficiently fine uniform grids, i.e., $hp<1$.}
Note that, since
\[
S_{p,h_\ell}(\hat{\Omega}) \subseteq S_{p,k,h_\ell}(\hat{\Omega}), 
\]
for all $0 \leq k < \revB{p-1}$, \cref{lem:2} is also valid in that case.

We also introduce, as in~\cite[Theorem 7.1]{takacs16}, \revB{and~\cite[Section 8]{sande19}}, a suitable subspace of $S_{p,h_{\ell}}(\hat{\Omega})$ to derive inverse inequalities.
\begin{defin}
\label{def:3}
The space of $q$-vanishing derivative univariate B-splines is defined by
\[
\tilde{S}_{p,h_{\ell}}^{q}(\hat{\Omega}) := \left\{ u_{p,h_{\ell}}\in S_{p,h_{\ell}}(\hat{\Omega}) \; \middle|\;  \frac{\partial^{2l+q}}{\partial x^{2l+q}}u_{p,h_{\ell}}(0)=\frac{\partial^{2l+q}}{\partial x^{2l+q}}u_{p,h_{\ell}}(1)=0 \quad \text{for all} \; l\in \mathbb{N}_0 \; \text{with}~2l+q<p\right \}
\]
\end{defin}
Note that according to \cite[Theorem 7.1]{takacs16}, \cref{lem:2} also holds for the space $\tilde{S}_{p,h_{\ell}}^{q}(\hat{\Omega})$. The following inverse inequality is a result from~\cite[Theorem 7.2]{takacs16}, \revB{or~\cite[Theorem 9.1]{sande19}.}
\begin{lem}
\label{lem:12}
For all mesh levels $\ell \in \mathbb{N}$, each $p \in \mathbb{N}_0$ and each $q \in \mathbb{N}_0$ with $0 \leq q \leq p$, 
\[
|u_{p,h_\ell}|_{H^q(0,1)} \leq C_2(q)h_\ell^{-q}\|u_{p,h_\ell}\|_{L^2(0,1)}
\]
is satisfied for all spline approximation $u_{p,h_\ell}\in \tilde{S}_{p,h_{\ell}}^{q}(\hat{\Omega})$
with $C_2(q) = (2\sqrt{3})^q$.
\end{lem}

%----------------------------------------------------------------------------------------
\subsection{Anisotropic $d$-dimensional B-spline spaces}
We extend the results of the previous section to anisotropic meshes in multiple dimensions, where $d\geq 2$. The tensor-product parameter domain is $\hat{\Omega}=(0,1)^d$. 

\begin{defin}
Let $\bm{\ell}=(\ell_1,\ldots,\ell_d) \in \mathbb{N}^d$ be the mesh level in each direction. The anisotropic $d$-dimensional B-spline space is defined by the tensor products of the univariate spaces:
\[
S_{p,h_{\bm{\ell}}}(\hat{\Omega}) := \bigotimes_{i=1}^d S_{p,h_{\ell_i}}(0,1).
\]
\end{defin}
Here, we consider B-splines of the same degree $p$ in each direction, while allowing the mesh size to vary along different directions. 

To extend the notion of projection introduced in the previous section to this $d$-dimensional space, we begin by defining the projection along a single direction.
\begin{defin}%[Unidimensional projection]
The $H^r$-orthogonal projection onto the spline space $S_{p,h_{\ell_i}}(0,1)$ in the single direction $i\in \{1,\ldots,d\}$ is defined by
\[
\Pi_{p,h_{\ell_{i}}} := \operatorname{I} \otimes \cdots \otimes  \pi_{p,h_{\ell_i}} \otimes  \cdots \otimes \operatorname{I} 
\]
\end{defin} 
Note that the unidimensional projections are commutative, i.e., for $i_1,i_2\in \{1,\ldots,d\}$, 
\[
\Pi_{p,h_{\ell_{i_1}}} \circ \Pi_{p,h_{\ell_{i_2}}} = \Pi_{p,h_{\ell_{i_2}}} \circ \Pi_{p,h_{\ell_{i_1}}}.
\]

We introduce the projection onto the B-spline space in a subset of dimension indices by successively applying the unidimensional projection.
\begin{defin}%[$k$-dimensional projection]
Let $k\in\mathbb{N}_0$ with $0 \leq k \leq d$ and $J \subset \{1,\dots,d\}$ be a subset of cardinality $|J|=k$.  The $k$-dimensional projection onto the B-spline space in the directions $J$ is defined by 
\[
\Pi_{p,h_{\bm{\ell}}}^{\,J} u 
  \;:=\;
  \Bigl(\,\prod_{i\in J}^{\circ}\Pi_{p,h_{\ell_i}}\,\Bigr) u.
\]
We also define the complementary projection by:
\[
(\operatorname{I}-\Pi_{p,h_{\ell}})^{\,J} u
  \;:=\;
  \Bigl(\,\prod_{i\in J}^{\circ} \bigl(\operatorname{I}-\Pi_{p,h_{\ell_i}}\bigr)\,\Bigr) u.
\]
\end{defin}
Here the notation $\prod_{i\in J}^{\circ}$ indicates composition of the operators with respect to the indices in $J$, the order being irrelevant since the operators commute. 

Based on these definitions, the projection of a function $u$ onto the $d$-dimensional spline space $S_{p,h_{\bm{\ell}}}(\hat{\Omega})$ can be defined as the composition of the univariate $H^r$-orthogonal projections applied successively along each coordinate direction, i.e.,
\begin{align}
\label{eq:15}
\Pi_{p,h_{\bm{\ell}}} u := \Pi_{p,h_{\bm{\ell}}}^{\{1,\ldots,d\}} u.
\end{align}

We introduce a first error estimate for the $d$-dimensional B-spline spaces. \revB{This result is a direct consequence of~\cite[Theorem 8.1]{takacs16}, or~\cite[Corollary 3.1]{sande19}, extended to spaces with anisotropic mesh sizes.}
\begin{prop}
\label{prop:3}
Let $u \in H^q_{\mathrm{mix}}(\hat{\Omega})$ with $q \in \mathbb{N}_0$. For all mesh levels $\bm{\ell} \in \mathbb{N}^d$, each $p\in\mathbb{N}_0$, and any $r \in \mathbb{N}$ such that $0 \leq r \leq q \leq p+1$,
\[
\bigl\|(\operatorname{I}-\Pi_{p,h_{\bm{\ell}}}) \,u
\bigr\|_{H^r(\hat{\Omega})}
\leq
C_3(d,q,r)
\Big(\sum_{i=1}^d h_{\ell_i}^{q-r}\Bigr)
\|u\|_{H^q(\hat{\Omega})}
\]
is satisfied with $\displaystyle C_3(d,q,r)=(r+1)^{d/2} \left(\revB{\frac{1}{\pi}}\right)^{q-r}$.
\end{prop}

The result of \cref{prop:3} is not suitable for proving sharp approximation properties of the sparse-grid spaces, as it does not provide sufficient information in the mixed directions. 

In the remainder of this section, we introduce three lemmas that will be used later to study the approximation properties of these spaces. First, in \cref{lem:7}, we decompose the $d$-dimensional projection defined in \cref{eq:15} into a telescopic sum of partial projections over subsets of coordinate directions. 
This decomposition will allow us to exploit cancellations of coarse errors in the sparse-grid combination technique. 

\begin{lem}
\label{lem:7}
For all mesh levels $\bm{\ell}\in\mathbb{N}^d$ and all $u\in L^2(\hat{\Omega})$ the following decomposition holds true:
\begin{align}
\label{eq:10}
(\operatorname{I}-\Pi_{p,h_{\bm{\ell}}})u = \sum_{k=1}^d \sum_{\substack{J\subset \{1,\ldots,d\}\\ |J|=k}}(-1)^{k-1} (\operatorname{I}-\Pi_{p,h_{\bm{\ell}}})^Ju.
\end{align}
\end{lem}

To prove \cref{lem:7}, we introduce the combinatorial result of \cref{lem:1}.
\begin{lem}
	\label{lem:1}
	For all $i \in \mathbb{N}_0$ such that $0\leq i\leq d-2$, we have:
	\begin{align}
	\label{lem:3:eq:1}
	\sum_{l=0}^{d-1}(-1)^l\binom{d-1}{l}l^i=0.
	\end{align}
\end{lem}

In \cref{lem:5}, we give an estimate of the approximation error of each B-spline space defined by the partial projections appearing in the decomposition \cref{eq:10}.
\begin{lem}%[Estimation of $k$-dimensional projection errors]
\label{lem:5}
Let $u \in H^q_{\mathrm{mix}}(\hat{\Omega})$ with $q \in \mathbb{N}_0$. For all mesh levels $\bm{\ell} \in \mathbb{N}^d$, each $p\in \mathbb{N}_0$, each $k\in\mathbb{N}$ with $1 \leq k\leq d$, each index set $J \subset \{1,\dots,d\}$ with $|J|=k$ and each $r \in \mathbb{N}$ such that $0 \leq r \leq q \leq p+1$,
\[
\bigl\|(\operatorname{I}-\Pi_{p,h_{\bm{\ell}}})^{J} u
\bigr\|_{H^r_{\mathrm{mix}}(\hat{\Omega})}
\leq
C_3(d,k,q,r)
\Big(\prod_{i\in J} h_i^{q-r}\Bigr)
\|u\|_{H^q_{\mathrm{mix}}(\hat{\Omega})}
\]
is satisfied with $\displaystyle C_3(d,k,q,r)=(r+1)^{d/2}\left(\revB{\frac{1}{\pi}}\right)^{k(q-r)}$.
\end{lem}

%------------------------------------------------------------------------------------
\section{Sparse-grid tensor-product approximation spaces}
\label{sec:3}
In this section we introduce finite-dimensional approximation spaces defined  in the tensor-product parameter domain $\hat{\Omega}=(0,1)^d$ by the sparse-grid tensor products of the unidimensional spaces $S_{p,h_\ell}(0,1)$.

We first introduce the so-called full tensor-product approximation space, constructed by taking uniform mesh size in all directions. Let $n\in \mathbb{N}$ be the maximum mesh level, i.e., $h_{\ell_i}\leq h=2^{-n}$, for all $i=1,\ldots,d$.
\begin{defin}
The full tensor-product approximation space is defined by 
\[
S_{p,h}^{(\infty)}(\hat{\Omega}) := S_{p,h_{\bm{\ell}}}(\hat{\Omega}) \quad \text{where}~ \ell_i = n, ~~ \forall i=1,\ldots,d.
\]
\end{defin}

In \cref{sec:3.1}, we introduce a $d$-dimensional sparse-e.g.grid approximation space, denoted by $S_{p,h}^{(1)}(\hat{\Omega})$, which is a subspace of the full tensor-product space:
\[
S_{p,h}^{(1)}(\hat{\Omega}) \subset S_{p,h}^{(\infty)}(\hat{\Omega}).
\]
The construction proceeds by expressing the full tensor product space as a direct sum of hierarchical subspace increments, and by neglecting the less significant contributions under the assumption that the target functions possess sufficiently smooth mixed derivatives. In this setting, the approximation acts on the space $H^q_{\mathrm{mix}}(\hat{\Omega})$ rather than on $H^q(\hat{\Omega})$.  \revC{The superscript $(1)$ indicates that the levels are selected according to their discrete $\ell^1$ norm, as opposed to the full-grid space where the $\ell^\infty$ norm is considered}.

 \revB{Hierarchical B-splines, defined as a hierarchy of locally refined meshes with nested knot vectors, have been initially introduced in~\cite{forsey88} to perform local refinement of tensor-product splines; while the approximation properties of hierarchical quasi-interpolants have been studied, e.g., in~\cite{speleers16,speleers17}. Hierarchical B-splines are however used differently in this paper. Our aim is to define a hierarchy of meshes globally refined so that an \emph{a priori} sparse-grid tensor-product truncation can be applied to reduce the computational costs.}

In \cref{sec:3.2}, we introduce a second approximation space, denoted by $S_{p,h}^{(C)}(\hat{\Omega})$, based on the sparse-grid combination technique, which is an approximation of hierarchical-based sparse-grid methods that is known, under suitable regularity assumptions, to achieve the same approximation order~\cite{griebel92,griebel12,pflaum97}.
 %and first proposed in~\cite{BECK2018128}. 

\subsection{Hierarchical subspace decomposition}
\label{sec:3.1}
A key element in the construction of sparse-grid tensor-product spaces is the definition of hierarchical subspaces, also called hierarchical increments, yielding a direct sum decomposition.  Specifically, we want to express the anisotropic approximation spaces as the direct sum decomposition
\begin{align}
\label{eq:4}
S_{p,h_{\bm{\ell}}}(\hat{\Omega}) = \bigoplus_{\bm{k} \leq \bm{\ell}} W_{p,h_{\bm{k}}}(\hat{\Omega}),
\end{align}
where the multi-index notation $\bm{k} \leq \bm{\ell}$ stands for component-wise inequalities, i.e., $k_j \leq \ell_j$ for all $j = 1, \dots, d$.  

A crucial property to obtain \cref{eq:4} is that the B-spline spaces are nested under refinement via knot insertion:
\[
S_{p,h_{\bm{\ell}-\bm{e}_j}}(\hat{\Omega}) \subset S_{p,h_{\bm{\ell}}}(\hat{\Omega}), \quad \forall j = 1, \dots, d,
\]
where $\bm{e}_j$ denotes the unit vector in the $j$-th direction. 
Since knot insertion preserves nestedness, every new function introduced during refinement is linearly independent of the previous ones. The hierarchical subspaces are then defined as
\begin{align}
\label{eq:14}
W_{p,h_{\bm{k}}}(\hat{\Omega}) := S_{p,h_{\bm{k}}}(\hat{\Omega}) /\bigoplus_{j=1}^d S_{p,h_{\bm{k}-\bm{e}_j}}(\hat{\Omega}),
\end{align}
where
\[
S_{p,h_{\bm{k}}}(\hat{\Omega}) := 0 \quad \text{if}~ \exists j\in\{1,\ldots,d\}  ~~s.t.~ k_j=-1.
\]
By defining the hierarchical increments as in \cref{eq:14}, \cref{eq:4} is satisfied. 
Intuitively, the hierarchical increments $W_{p,h_{\bm{k}}}(\hat{\Omega}) $ spans the new basis functions introduced at the refinement step. 

Using the definition of these hierarchical subspaces, we introduce the hierarchical-based sparse-grid approximation space. 

\revB{\begin{defin}
\label{def:1}
The hierarchical sparse-grid approximation space is defined by
\[
S_{p,h}^{(1)}(\hat{\Omega}) := \bigoplus_{\bm{\ell} \in \mathscr{H}_h} W_{p,h_{\bm{\ell}}}(\hat{\Omega}),
% \Bigl\{ u \in H^p_{\mathrm{mix}}(\hat{\Omega}) \;\Big|\; u = \sum_{\bm{\ell}\in \mathscr{H}} w_{p,h_{\bm{\ell}}}, \quadw_{p,h_{\bm{\ell}}} \in W_{p,h_{\bm{\ell}}}(\hat{\Omega}) \Bigr\},
\]
where 
\[
\mathscr{H}_h = \left \{\bm{\ell} \in \mathbb{N}^d ~~ | ~~ |\bm{\ell}|_1 \leq |\log h| + (d-1), \quad  \ell_j \geq 1\right\}.
\]
\end{defin}}

\subsection{Sparse-grid combination technique}
\label{sec:3.2}
The sparse-grid combination technique is an alternative formulation for representing functions from hierarchical sparse-grid spaces, such as $S_{p,h}^{(1)}(\hat{\Omega})$. In general, it yields an approximation of $H^q_{\mathrm{mix}}(\hat{\Omega})$ of the same order of accuracy \cite{griebel92,pflaum97}.  The method relies on forming a linear combination of contributions that live in standard tensor-product approximation spaces defined with admissible levels $\bm{\ell}$ satisfying $|\bm{\ell}|_1 \approx |\log h|$. 

Specifically, the admissible levels are such that $\bm{\ell}\in \mathscr{L}_h$. We introduce the index set of admissible levels as
\begin{align}
\label{eq:3}
\mathscr{L}_h := \bigcup_{0 \leq l \leq d-1} \mathscr{L}_{h,l}, \quad \text{where}~
\mathscr{L}_{h,l} := \left\{ \bm{\ell} \in \mathbb{N}^d \;\middle|\; 
|\bm{\ell}|_1 = |\log h| + d-1- l, \ \ell_j \geq 1 \right\}.
\end{align}
We define the sparse-grid combination technique approximation space by combining the spaces $S_{p,h_{\bm{\ell}}}(\hat{\Omega})$ associated with the admissible levels.
\begin{defin}
\label{def:2}
The sparse-grid combination technique approximation space is defined by
\[
S_{p,h}^{(C)}(\hat{\Omega}) 
:= \Bigl\{ u \in H^p_{\mathrm{mix}}(\hat{\Omega}) \;\Big|\; 
u = \sum_{\bm{\ell} \in \mathscr{L}_h} c_{\bm{\ell}} u_{p,h_{\bm{\ell}}}, \quad
u_{p,h_{\bm{\ell}}} \in S_{p,h_{\bm{\ell}}}(\hat{\Omega}) \Bigr\},
\]
with the combination coefficients given by
\[
c_{\bm{\ell}} := (-1)^l \binom{d-1}{l}, \quad \text{for}\; \bm{\ell} \in \mathscr{L}_{h,l}.
\]
\end{defin}

We extend the definition of the $d$-dimensional projection introduced in \cref{eq:15} to this space by defining the projection as the combination of projections associated with the admissible levels:
\[
\Pi_{p,h_n}^{(C)} := \left[ \sum_{\bm{\ell} \in \mathscr{L}_h} c_{\bm{\ell}} \, \Pi_{p,h_{\bm{\ell}}} \right].
\]

\subsection{Equivalence between spaces}
In the previous sections, we have introduced two approximation spaces based on sparse-grid tensor products of univariate B-spline spaces. 
We now establish that the two approximation spaces are identical.

\begin{thm}
The hierarchical subspace decomposition and the combination technique formulations yield identical spaces, i.e.,
\[
S_{p,h}^{(1)}(\hat{\Omega}) = S_{p,h}^{(C)}(\hat{\Omega}).
\]
\label{thm:2}
\end{thm}

To prove this theorem, that is provided in the appendix,~\cref{apd:1}, we have to introduce two auxiliary lemmas. \cref{lem:4} establishes that the sum of the combination coefficients over the admissible levels equals one, while \cref{lem:3} is a combinatorial result used in the proof of \cref{lem:4}.

\begin{lem}
	\label{lem:3}
	Let $\ell,k\in\mathbb{N}$ be such that $0\leq \ell \leq |\log h|$ and $0\leq k \leq d-1$, then:
	\begin{align}
	\label{lem:3:eq:0}
	\underset{l=0}{\overset{d-1}{\sum}}(-1)^l\binom{d-1}{l}\binom{|\log h|+d-2-l-\ell}{d-1-k}=\left\{\begin{array}{ll}
	1 \quad \text{if  } k=0, \\
	0 \quad \text{else.}
	\end{array}\right.
	\end{align}
\end{lem}

\begin{lem}%[Unity sum of combination coefficients]
\label{lem:4}
The combination coefficients verify
\begin{align}
\sum_{\bm{\ell}\in \mathscr{L}_h} c_{\bm{\ell}} = 1.
\end{align}
\end{lem}

As a result of \cref{thm:2}, the two spaces and thus their approximation properties are the same. 
We shall then derive properties of this space using either the combination technique or the hierarchical formulation, choosing the approach that makes a particular proof more straightforward. 

\subsection{Error estimate and inverse inequality}

%However, although the two spaces are mathematically identical, their computational properties may differ due to the distinct constructions. 
%In particular when applied to Galerkin methods, the size, condition number, and sparsity patterns of the resulting discretization matrices, such as the mass and stiffness matrices, are not the same. 
%A detailed investigation of these computational aspects is deferred to future work.

\subsubsection{Approximation error estimate}
In this section, we provide an approximation error estimate in mixed Sobolev norm for the sparse-grid space. To prove the result, we consider the combination technique formulation of the sparse-grid space.
Before stating the main result of this section, we introduce \cref{lem:8} expliciting the cancellations of the coarse contributions in the combination technique. 

\begin{lem}%[Cancelation of errors]
\label{lem:8}
Let $k\in\mathbb{N}$ with $1 \leq k \leq d$, and $J \subset \{1,\dots,d\}$ be an index set of cardinality $|J|=k$.  
Suppose $\{a_{J,\bm{\ell}}\}_{J\subset\{1,\dots,d\}}$ is a family of $d$-dimensional functions depending on the directions indexed by $J$.  
Then the following identity holds:
\begin{align}
\sum_{\bm{\ell}\in \mathscr{L}} c_{\bm{\ell}}
  \sum_{k=1}^d \; \sum_{\substack{J \subset \{1,\ldots,d\}\\ |J|=k}}a_{J,\bm{\ell}}
&=
\sum_{k=1}^{d-1}\;\sum_{l=0}^{d-2}
   C_3(d,k,l)\,
   \sum_{\substack{J \subset \{1,\ldots,d\}\\ |J|=k}}
                   \sum_{\bm{\ell}\in \mathscr{L}^J_{l,k}}
       a_{J,\bm{\ell}} \label{lem:8:eq:1}\\
&\quad+\;\sum_{l=0}^{d-1}
   C_4(d,l)\,
   \sum_{\bm{\ell}\in \mathscr{L}_l}
       a_{\{1,\ldots,d\},\bm{\ell}}, \nonumber
\end{align}
where the notation
\[
\mathscr{L}^J_{h,l,k} := \Bigl\{\bm{\ell}\in \mathbb{N}^d \; \Big|\; \sum_{i\in J} |\ell_{i}|= |\log h|+k-1-l, ~~ \ell_{J_j}\geq 1, ~~ \text{for} ~ j=1,\ldots,k  \Bigr\}
\]
 is introduced. The constants are given by
\[
C_3(d,k,l) = \sum_{\kappa=0}^l (-1)^\kappa
   \binom{d-1}{\kappa}
   \binom{d-2k-1+l-\kappa}{\,d-1-k},
\qquad
C_4(d,l) = (-1)^l \binom{d-1}{l}.
\]
\end{lem}

With this lemma in hand, we can now state \cref{lem:6}, which is the main result of this section.
\begin{lem}
\label{lem:6}
Let $u\in H^q_{\mathrm{mix}}(\hat{\Omega})$ with $q\in\mathbb{N}_0$. For each $p\in\mathbb{N}_0$ and each $r\in \mathbb{N}$ with $0\leq r \leq q \leq p+1$,
\begin{align}
\label{eq:12}
\left\|(\operatorname{I}-\Pi_{p,h}^{(C)})\,u\right\|_{H_{\mathrm{mix}}^r(\hat{\Omega})} \leq C_{10}(d,q,r) h^{q-r} |\log h|^{d-1} \|u\|_{H_{\mathrm{mix}}^q(\hat{\Omega})}
\end{align}
is satisfied with
\[
C_{10}(d,q,r)= \frac{(d-1)^{d-1}}{(d-1)! (\log 2)^{d-1}}\sum_{l=0}^{d-2}2^{-(q-r)(d-1-l)}(-1)^l \binom{d-1}{l} (r+1)^{d/2}\left(\frac{1}{\pi}\right)^{d(q-r)},
\] 
a constant that does not depend on the mesh size $h$ or the degree $p$. 
\end{lem}

On tensor-product domains such as $\hat{\Omega}$, it is well established that sparse-grid approximations using piecewise linear functions achieve an $L^2$ approximation error of order $\mathcal{O}(h^2 |\log h|^{d-1})$, with a dependence on the $H^2_{\mathrm{mix}}$ Sobolev norm of the solution~\cite{bungartz04,griebel90}. This result has been generalized to piecewise Lagrange polynomials of degree $p$, for which the $L^2$ error scales as $\mathcal{O}(h^{p+1} |\log h|^{d-1})$ and depends on the $H^{p+1}_{\mathrm{mix}}$ Sobolev norm~\cite{bungartz04}.
\cref{lem:6} further extends these results to B-splines of maximal smoothness, i.e., B-splines of degree $p$ with regularity $H^p(0,1)$ in each direction. Considering \cref{lem:6} with $r=0$ and $q=p+1$, the $L^2$ approximation error for such B-splines exhibits the same scaling as for Lagrange polynomials. Notably, as in the case of isotropic full tensor-product spaces treated in~\cite{takacs16,sande19}, the approximation error is independent of the spline degree $p$.

\subsubsection{Inverse inequality}
\label{sec:inv_param}
In this section, we derive an inverse inequality in mixed Sobolev norms for the sparse-grid approximation space. 
To this end, we extend the definition of the unidimensional space of $q$-vanishing derivative univariate B-splines \cref{def:3} to the sparse-grid spaces. 
Specifically, we define the spaces $\tilde{S}_{p,h}^{q,(1)}(\hat{\Omega})$ and $\tilde{S}_{p,h}^{q,(C)}(\hat{\Omega})$ by replacing the univariate spaces $S_{p,h_\ell}(0,1)$ with their counterparts $\tilde{S}_{p,h_\ell}^{q}(0,1)$ in \cref{def:1} and \cref{def:2}, respectively.

Using the same arguments as in the proof of \cref{thm:2}, we can demonstrate that these two spaces are identical. Moreover, according to~\cite[Theorem 7.1]{takacs16}, the estimate \cref{eq:12} of \cref{lem:6} also holds for spaces constructed with $q$-vanishing-derivative B-splines.

\begin{lem}
\label{lem:10}
For each $0\leq q \leq p$, the inverse inequality
\begin{align}
\|u_{p,h}^{(1)}\|_{H^{q}_{\mathrm{mix}}(\hat{\Omega})} \leq C_{11}(d,q) h^{-q} |\log h|^{d/2} \|u_{p,h}^{(1)}\|_{L^2(\hat{\Omega})}
\end{align}
is statisfied for all $u_{p,h}^{(1)}\in \tilde{S}_{p,h}^{q,(1)}(\hat{\Omega})$, with
\begin{align}
C_{11}(d,q)=  (q+1)^{d/2} (2\sqrt{3})^{dq}\frac{2^{(d-1)} 2^{d/2}}{d! } ,
\end{align} 
a constant that does not depend on the mesh size $h$ or the degree $p$. 
\end{lem}
To prove this result, we consider the hierarchical sparse-grid formulation. We introduce \cref{lem:11}, an auxiliary result required to prove \cref{lem:10}.
 \begin{lem}
\label{lem:11}
Let $\bm{s}\in \mathbb{N}^d$ with $0\leq s_1,\ldots,s_d \leq q \leq p$, and for all mesh level $\bm{\ell}\in \mathscr{H}_h$, each $p\in \mathbb{N}_0$,  $w_{p,h_{\bm{\ell}}}\in W_{p,h_{\bm{\ell}}}(\hat{\Omega})$, then
\begin{align}
\Bigl\|\frac{\partial^{s_1}}{\partial x_1^{s_1}}\ldots\frac{\partial^{s_d}}{\partial x_d^{s_d}}w_{p,h_{\bm{\ell}}}\Bigr\|_{L^2(\hat{\Omega})} \leq(2\sqrt{3})^{|\bm{s}|_1}  \Bigl(\prod_{i=1}^d h_{\ell_i}^{-s_i}\Bigr)\|w_{p,h_{\bm{\ell}}}\|_{L^2(\hat{\Omega})}.
\end{align}
\end{lem}

%-------------------------------------------------------------------------------------------
\section{Non-tensor-product approximation spaces}
\label{sec:iga}
We now extend the previous results to more general non-tensor product geometric domains.  

Let the physical domain be defined by a smooth mapping
\[
\bm{F}:\hat{\Omega} \rightarrow \Omega,
\]
which is a diffeomorphism between the tensor-product parameter domain $\hat{\Omega}$ and the physical domain $\Omega$, i.e.,
\[
\Omega := \bm{F}(\hat{\Omega}).
\]
The mapping $\bm{F}$ provides a geometric parametrization of the physical domain.

Given meshes of the parameter and physical domains, denoted respectively by $\hat{\Omega}_{h_{\bm{\ell}}}$ and $\Omega_{h_{\bm{\ell}}}$, and an approximation space $S_{h_{\bm{\ell}}}(\hat{\Omega})$ defined on the parameter domain, we define the corresponding approximation space in the physical domain as
\[
S_{h_{\bm{\ell}}}(\Omega) :=
\left\{
u_{h_{\bm{\ell}}}:\Omega \to \mathbb{R}
\ \middle|\
u_{h_{\bm{\ell}}} = \hat{u}_{h_{\bm{\ell}}} \circ \bm{F}_{h_{\bm{\ell}}}^{-1}, \quad
\hat{u}_{h_{\bm{\ell}}} \in S_{h_{\bm{\ell}}}(\hat{\Omega})
\right\}.
\]
Here, $\bm{F}_{h_{\bm{\ell}}}$ denotes a $\mathcal{C}^1$-diffeomorphism that maps the parameter mesh $\hat{\Omega}_{h_{\bm{\ell}}}$ to the corresponding physical mesh $\Omega_{h_{\bm{\ell}}}$:
\[
\Omega_{h_{\bm{\ell}}} := \bm{F}_{h_{\bm{\ell}}}(\hat{\Omega}_{h_{\bm{\ell}}}).
\]

Several strategies exist in the context of sparse grids to define $\bm{F}_{h_{\bm{\ell}}}$, and consequently to construct the spaces $S_{h_{\bm{\ell}}}(\Omega)$. 
In~\cite{Achatz:2003aa}, $\bm{F}_{h_{\bm{\ell}}}$ is defined as the interpolant of $\bm{F}$ evaluated at the nodes of $\Omega_{h_{\bm{\ell}}}$. 
In~\cite{bungartz98,Dornseifer:1996aa}, curvilinear meshes are employed and $\bm{F}_{h_{\bm{\ell}}}$ is defined by transfinite interpolation. 
In~\cite{BECK2018128}, $\bm{F}_{h_{\bm{\ell}}}$ is represented using B-splines or NURBS (rational B-splines with associated weights), where the geometry is defined as a linear combination of B-spline basis functions weighted by control points. 

In this work, we adopt the isogeometric analysis representation, that is, we define $\bm{F}_h$ as a B-spline mapping. We restrict ourselves to B-spline functions for simplicity; however, the extension to NURBS functions, as in~\cite{bazilev11}, could also be considered.
In practical applications, the geometry of the physical domain is typically described on a coarse mesh with relatively few elements, while the solution is approximated on successively finer meshes to achieve the desired accuracy. 
This is possible because B-spline or NURBS mappings can represent the geometry exactly, even on the coarsest mesh.

We assume that the geometry is described \revB{by a coarse parameter mesh $\hat{\Omega}_{h_0}$ associated with a fixed coarse mesh size satisfying
\[
h_0 > h.
\]
For any mesh size $h_{\bm{\ell}}$, the physical domain is then defined by
\[
\Omega_{h_{\bm{\ell}}} := \bm{F}_{h_0}(\hat{\Omega}_{h_0}),
\]
where the geometric mapping is given by
\[
\bm{F}_{h_0}(\xi_1,\ldots,\xi_d)
=
\sum_{i_1=1}^{n_1} \cdots \sum_{i_d=1}^{n_d}
\bm{P}_{i_1,\ldots,i_d}
\, s_{p,h_0,i_1}(\xi_1)\cdots s_{p,h_0,i_d}(\xi_d).
\]
}
Here, $s_{p,h_0,i_j} \in S_{p,h_0}(0,1)$ denote the univariate B-spline basis functions defined on the coarse mesh with $n_j=(h_0)_j^{-1}$ for $j=1,\ldots,d$, and $\bm{P}_{i_1,\ldots,i_d}\in\mathbb{R}^d$ are the control points defining the geometric mapping.

Following~\cite{bazilev11,HUGHES20054135}, when the approximation spaces are refined, the control points are adjusted so that the mapping $\bm{F}_{h_{\bm{\ell}}}$ remains unchanged. 
Thus, the geometry and its parametrization are kept fixed as the mesh is refined, i.e., when considering $\hat{\Omega}_{h_{\bm{\ell}}}$ for increasing $\bm{\ell}$. 
Specifically,
\[
\bm{F}_{h_{\bm{\ell}}} = \bm{F}_{h_0},
\]
where $\bm{F}_{h_{\bm{\ell}}}$ is defined with control points $\bm{P}^{(h_{\bm{\ell}})}_{i_1,\ldots,i_d}$ that are adjusted according to $h_{\bm{\ell}}$.

We define the isogeometric analysis approximation space by composing the parameter space with the geometric mapping $\mathbf{F}_{h_0}$. This composition is also called the push-forward operation of the parameter space.
\begin{defin}
The physical domain sparse-grid approximation space is defined by 
\[
S_{p,h}^{(1)}(\Omega) := \Bigl\{ u \in H^p_{\mathrm{mix}}(\Omega) \;\Big|\; 
u = u_{p,h}^{(1)} \circ \mathbf{F}_{h_0}^{-1}, \quad
u_{p,h}^{(1)} \in S_{p,h}^{(1)}(\hat{\Omega}) \Bigr\}.
\]
\end{defin}
Note that according to \cref{thm:2}, this space can be defined equivalently using the sparse-grid combination technique.

\subsection{Error estimate and inverse inequality}
\subsubsection{Approximation error estimate}
In this section we extend the results from the parameter approximation space to the physical domain approximation space.
To prove the main result of this paper,~\cref{thm:1}, we introduce a lemma providing bounds for the composition of a function with the geometric mapping $\bm{F}_{h_0}$, or its inverse $\bm{F}_{h_0}^{-1}$. 

\begin{lem}%[Composition with geometric mapping]
\label{lem:9}
For all $u\in H^{dq}(\Omega)$, with $0\leq q \leq p+1$, it holds that
\begin{align}
&|u|_{H^q(\Omega)} \leq \sum_{m=0}^{q} C_{12} |u\circ \bm{F}_{h_0}|_{H^m(\hat{\Omega})}, \label{lem:9:eq:1} \\
&\|u\circ \bm{F}_{h_0}\|_{H^q_{\mathrm{mix}}(\hat{\Omega})} \leq \sum_{m=0}^{dq} C_{13}|u|_{H^m(\Omega)} ,\label{lem:9:eq:2}
\end{align}
where 
\begin{align*}
&C_{12}=C_{12}(m,q,\|\bm{\nabla F}_{h_0}^{-1}\|_{W^{q,\infty}(\Omega)}, \|\det \bm{\nabla} \bm{F}_{h_0}\|_{L^\infty(\hat{\Omega})}), \\
&C_{13}=C_{13}(m,q,\|\bm{\nabla F}_{h_0}\|_{W^{q,\infty}(\hat{\Omega})}, \|\det \bm{\nabla} \bm{F}_{h_0}^{-1}\|_{L^\infty(\Omega)}),
\end{align*}
are constants that depend in particular on the derivatives up to order $q$ of the geometric mapping $\bm{F}_{h_0}$ and its inverse $\bm{F}_{h_0}^{-1}$.
\end{lem}

This result is an extension of~\cite[Lemma 3.5]{bazilev11}, \revB{or~\cite[Lemma 12]{sande19-1},} to functions of mixed Sobolev spaces; it uses an argument similar to the result from~\cite[Lemma 3]{ciarlet72}. 
Note that, according to \cref{lem:9:eq:2}, in order to bound the $H^q_{\mathrm{mix}}(\Omega)$ norm of the composition of a function $u$ with the geometric mapping $\bm{F}_{h_0}$, it is necessary to control the derivatives of $u$ up to order $dq$ in each direction. 
This arises because the mapping $\bm{F}_{h_0}$ is not a separable function. 
Although $\bm{F}_{h_0}$ is expressed in a tensor-product form, its control points couple the contributions of the individual basis functions, each of which is separable on its own. 
As a result, the mixed derivatives of the composed function $D^{\bm{\alpha}}(u \circ \bm{F}_{h_0})$ with $|\bm{\alpha}|_{\infty} = q$, appearing in the $H^q(\Omega)$ seminorm, do not vanish, thereby destroying the coordinatewise independence.

\revB{The first main result of the paper is~\cref{thm:1}. This result extends~\cite[Theorems 7.3 and 8.1]{takacs16} and~\cite[Theorem 8]{sande19-1} to the sparse-grid approximation space. In~\cref{thm:1}, the constant is independent of the mesh size $h$ and the spline degree $p$, and is given by
\[
C_{14}(d, q, r, \|\bm{\nabla F}_{h_0}^{-1}\|_{W^{q,\infty}(\Omega)},\|\bm{\nabla F}_{h_0}\|_{W^{q,\infty}(\hat{\Omega})}, \|\det \bm{\nabla} \bm{F}_{h_0}\|_{L^\infty(\hat{\Omega})}, \|\det \bm{\nabla} \bm{F}_{h_0}^{-1}\|_{L^\infty(\Omega)}).
\]
However, the constant depends on the shape of the physical domain and on the coarse mesh size $h_0$ used to define the geometric mapping.}

Note that the result from \cref{thm:1} differs from the convergence results obtained for sparse-grid approximations defined on tensor-product domains, such as $\hat{\Omega}$. 
In those cases, \cref{lem:6}, which can equivalently be expressed in terms of the seminorm instead of the full norm, shows that the $H^r_{\mathrm{mix}}(\hat{\Omega})$ seminorm of the interpolation error is controlled by the $H^q_{\mathrm{mix}}(\hat{\Omega})$ seminorm of the solution, for $0 \leq r \leq q$. 
However, for non-tensor-product domains constructed through a geometric mapping $\bm{F}$, controlling the $H^r(\Omega)$ seminorm of the interpolation error requires controlling the $H^{dq}(\Omega)$ norm of the solution, with $0 \leq r \leq q$. 
This implies that higher-order derivatives of the solution must be bounded up to order $dq$ in a single direction. 
Consequently, employing sparse-grid tensor products to construct approximation spaces for non-tensor-product domains necessitates stronger regularity assumptions, particularly on isotropic (non-mixed) derivatives, compared to standard methods.

\subsubsection{Inverse inequality}
\label{sec:inv}
We extend the inverse inequality of \cref{sec:inv_param} to the physical domain approximation space.
To this end, we introduce the spaces constructed from $q$-vanishing derivative B-splines.
Specifically, we define the space $\tilde{S}_{p,h}^{q,(1)}(\Omega)$ by taking the push-forward operation from the $q$-vanishing derivative B-spline parameter space. Thanks to \cite[Theorem 7.1]{takacs16}, we can prove that \cref{thm:1} is also valid for this space.

The second main result of this paper is the inverse inequality of~\cref{thm:3}. In this theorem, the constant does not depend on the mesh size $h$ or the degree $p$, and is given by
\[
C_{15}(q, \|\bm{\nabla F}_{h_0}^{-1}\|_{W^{q,\infty}(\Omega)},\|\bm{\nabla F}_{h_0}\|_{W^{q,\infty}(\hat{\Omega})}, \|\det \bm{\nabla} \bm{F}_{h_0}\|_{L^\infty(\hat{\Omega})}, \|\det \bm{\nabla} \bm{F}_{h_0}^{-1}\|_{L^\infty(\Omega)}).
\]

\subsection{Complexity and computational properties}
The main benefit of sparse-grid spaces is that they retain nearly the same approximation properties as standard tensor-product spaces, as demonstrated in \cref{thm:1}, while significantly reducing the number of degrees of freedom. 
This can be formalized by the following result concerning the dimension of the sparse-grid approximation space.
\begin{prop}
\label{prop:1}
The dimension of the sparse-grid B-spline approximation space satisfies
\[
\dim S_{p,h}^{(1)}(\Omega)  \lesssim h^{-1} |\log h|^{d-1}.
\]
\end{prop}
For comparison, consider the full tensor-product space obtained by refining the univariate spline spaces uniformly in all coordinate directions, denoted by $S_{p,h}^{(\infty)}(\Omega)$.
From \cref{prop:3} and \cref{lem:9}, for all $u \in H^q(\Omega)$, we have
\begin{align}
\|(\operatorname{I} - \Pi_{p,h}^{(\infty)})\, u\|_{H^r(\Omega)} 
\lesssim h^{q-r} \|u\|_{H^q(\Omega)},
\end{align}
where $\Pi_{p,h}^{(\infty)}$ denotes the $H^r$-orthogonal projection onto $S_{p,h}^{(\infty)}(\Omega)$. 

When sparse-grid spaces are employed instead of full tensor-product spaces, the approximation error deteriorates only mildly (by a factor of $|\log h|^{d-1}$) which becomes negligible as the mesh is refined ($h \to 0$). 
This mild deterioration is more than compensated by the significant reduction in the number of degrees of freedom. 
Indeed, the dimension of the full tensor-product space is given by
\begin{align}
\dim S_{p,h_n}^{(\infty)}(\Omega) = (h^{-1} + p)^d,
\end{align}
so the relative computational advantage of sparse grids increases rapidly with the spatial dimension~$d$.

\revA{In the previous sections, we have introduced two approaches to construct sparse-grid approximation spaces, using either a hierarchical decomposition or the sparse-grid combination technique. Although we prove in~\cref{thm:2} that these two constructions yield the same approximation space, they may exhibit different computational properties. For instance, when applied to a Galerkin method for the variational formulation of differential equations, the resulting mass and stiffness matrices differ significantly between the two approaches.

The number of basis functions required to represent a function in the space, and therefore the size of the resulting linear system, is smaller when using the hierarchical decomposition than when using the sparse-grid combination technique. Specifically, the number of basis functions used in the hierarchical decomposition representation is given by
\[
N^{(1)} = 1 + \sum_{k=1}^d \binom{d}{k} \sum_{i=k}^{|\log h|} \binom{i-1}{k-1} 2^{i-k},
\]
whereas the sparse-grid combination technique representation requires
\[
N^{(C)} = 2^{|\log h|} \sum_{k=0}^{d-1} \binom{|\log h| - 1 + k}{k} 2^k
\]
basis functions. 

Both quantities scale as $\mathcal{O}(h^{-1} |\log h|^{d-1})$, in agreement with the estimate of~\cref{prop:1}. However, $N^{(C)}$ is larger than $N^{(1)}$ by approximately a factor of $2^d$.
%A second major difference between the two approaches is the sparsity pattern of the matrices.
%The hierarchical decomposition uses a representation of functions based on basis functions with supports that may overlap between different levels, so that the discretization matrices contain dense blocks with a large number of nonzeros entries, as well as sparse blocks with few nonzeros entries. The sparse-grid combination technique approach is based on a combination of nodal basis representations. For each of these representations, the basis function supports overlap only with their closest neighbors, which results in sparse matrices with only few nonzeros entries.

Numerical experiments and a more detailed investigation of the computational aspects are addressed in a forthcoming paper~\cite{deluzet26}, in which the sparse-grid approximation space is applied within a Galerkin method combined with a particle method for the discretization of a kinetic partial differential equation arising in plasma physics.}

\section*{Acknowledgements}
\revC{The author thanks the anonymous reviewers for their suggestions, which helped to improve the quality of this paper.}
\bibliography{bib}
\bibliographystyle{plain}

%############################################################
%############################################################
\appendix 
\section{Proofs}
\label{apd:1}

\begin{proof}[Proof of \cref{lem:2}]
According to~\cite[Theorem 3.1]{sande19}, under the stated assumptions, there exists a spline approximation $u_{p,h_\ell}\in S_{p,h_\ell}(\hat{\Omega})$ such that: 
\[
|u-u_{p,h_\ell}|_{H^r(\hat{\Omega})} \leq \left( \revB{\frac{1}{\pi}}h_\ell\right)^{q-r} |u|_{H^q(\hat{\Omega})}.
\]
Applying the definition of the seminorm $H^r$-orthogonal projection, we obtain the result.
\end{proof}

\begin{proof}[Proof of \cref{lem:12}]
From \cite[Theorem 7.2]{takacs16}, by induction, we obtain the result.
\end{proof}

\begin{proof}[Proof of \cref{prop:3}]
The proof is strictly similar to~\cite[Theorem 8.1]{takacs16}, \revB{or~\cite[Corollary 3.1]{sande19}}, but considering different mesh sizes $h_{\ell_i}$, $i=1,\ldots,d$, in each direction.
\end{proof}

\begin{proof}[Proof of \cref{lem:7}]
By expanding each of the following products as
\[
(\operatorname{I}-\Pi_{p,h_{\bm{\ell}}})^Ju = \operatorname{I}+\sum_{l=1}^k(-1)^l\sum_{\substack{\tilde{J}\subset J\\ |\tilde{J}|=l}}\Pi_{p,h_{\bm{\ell}}}^{\tilde{J}},
\]
the right hand side of of \cref{eq:10} can be rewritten as
\begin{align}
 \sum_{k=1}^d \sum_{\substack{J\subset \{1,\ldots,d\}\\ |J|=k}}(-1)^{k-1} (\operatorname{I}-\Pi_{p,h_{\bm{\ell}}})^Ju &= \operatorname{I} \left[\sum_{k=1}^d \sum_{\substack{J\subset \{1,\ldots,d\}\\ |J|=k}}(-1)^{k-1}\right] \nonumber \\
&+ \sum_{k=1}^d\sum_{\substack{J\subset \{1,\ldots,d\}\\ |J|=k}}(-1)^{k-1}\sum_{l=1}^k(-1)^l\sum_{\substack{\tilde{J}\subset J\\ |\tilde{J}|=l}}\Pi_{p,h_{\bm{\ell}}}^{\tilde{J}}. \label{lem:7:eq:1}
\end{align}
The first term of the right hand side of \cref{lem:7:eq:1} is equal to $\operatorname{I}$ because:
\begin{align}
\sum_{k=1}^d \sum_{\substack{J\subset \{1,\ldots,d\}\\ |J|=k}}(-1)^{k-1} &= \sum_{k=1}^d(-1)^{k-1}\binom{d}{k}\label{lem:7:eq:2} \\
&= -\sum_{k=0}^d(-1)^{k}\binom{d}{k}-(-1)\label{lem:7:eq:3} \\
&=1.\label{lem:7:eq:4}
\end{align}
\cref{lem:7:eq:2} is obtained by counting the number of elements in the sum on $J\subset \{1,\ldots,d\}$, $|J|=k$; in \cref{lem:7:eq:3}, we have added and subtracted the term $k=0$ to the sum; in \cref{lem:7:eq:4}, we have used \cref{lem:3:eq:1} for $i=0$, which is a direct consequence of the binomial theorem. We recast the second term of the right hand side of \cref{lem:7:eq:1} by rearranging the sums as
\[
\sum_{k=1}^d\sum_{\substack{J\subset \{1,\ldots,d\}\\ |J|=k}}(-1)^{k-1}\sum_{l=1}^k(-1)^l\sum_{\substack{\tilde{J}\subset J\\ |\tilde{J}|=l}}\Pi_{p,h_{\bm{\ell}}}^{\tilde{J}}= \sum_{l=1}^d\underbrace{\sum_{k=l}^d(-1)^{l+k-1}\sum_{\substack{J\subset \{1,\ldots,d\}\\ |J|=k}}\sum_{\substack{\tilde{J}\subset J\\ |\tilde{J}|=l}}\Pi_{p,h_{\bm{\ell}}}^{\tilde{J}}}_{P(l)},
\]
and we look at each term $P(l)$, for $1\leq l \leq d$, of the outermost sum. First for $l=d$, we directly obtain:
\[
P(d)=-(\Pi_{p,h_{\bm{\ell}}})^{\{1,\ldots d\}}.
\]
For $1\leq l \leq d-1$, we rearrange the sums appearing in $P(l)$ as
\begin{align*}
P(l)&=\sum_{\substack{\tilde{J}\subset \{1,\ldots,d\}\\ |\tilde{J}|=l}}\Pi_{p,h_{\bm{\ell}}}^{\tilde{J}}\sum_{k=l}^d(-1)^{l+k-1}\binom{d-l}{k-l} \\
&= 0.
\end{align*}
We obtain the last equation because the innermost sum vanishes. Indeed, by substituting $\kappa = k-l$, we have:
\[
\sum_{k=l}^d(-1)^{l+k-1}\binom{d-l}{k-l} = -\sum_{\kappa=0}^{d-l}(-1)^{\kappa}\binom{d-l}{\kappa}=0.
\]
\end{proof}

\begin{proof}[Proof of \cref{lem:1}]
We demonstrate the result by induction.
For $i=0$, the relation is immediate using the binomial theorem. We assume that \cref{lem:3:eq:1} holds for $0\leq i \leq d-3$, then we have:
\begin{align}
\sum_{l=0}^{d-1}(-1)^l\binom{d-1}{l}l^{i+1}&=\sum_{l=1}^{d-1}(-1)^l\binom{d-1}{l}l^{i+1} \label{lem:3:eq:6}\\
&=(d-1)\sum_{l=1}^{d-1}(-1)^l\binom{d-2}{l-1}l^i \label{lem:3:eq:2}\\ 
&=-(d-1)\sum_{l=0}^{d-2}(-1)^l\binom{d-2}{l}(l+1)^i \nonumber\\
&=-(d-1)\sum_{j=0}^{i}\binom{i}{j}\sum_{l=0}^{d-2}(-1)^l\binom{d-2}{l}l^j \label{lem:3:eq:4} \\
&= 0. \label{lem:3:eq:5}
\end{align}
In the above, we have used the relation
\[
\dbinom{d-1}{l}= \dfrac{d-1}{l}\dbinom{d-2}{l-1}
\]
in \cref{lem:3:eq:2}; the binomial theorem in \cref{lem:3:eq:4}; and the induction hypothesis in \cref{lem:3:eq:5}.
\end{proof}

\begin{proof}[Proof of \cref{lem:5}]
We demonstrate the result in two dimensions, for $d=2$. The generalization to higher dimensions is straightforward and uses the same arguments. 
%We assume $u\in \mathcal{C}^{\infty}(\hat{\Omega})$ and conclude using a density argument, since $\mathcal{C}^{\infty}(\hat{\Omega})$ is dense in $H^q_{\mathrm{mix}}(\hat{\Omega})$. 

We start by estimating the term corresponding to the projection in both directions, i.e., $k=2$ and $J=\{1,2\}$. Let $s,m\in \mathbb{N}$ be such that $0\leq s,m \leq r$. Since $\operatorname{I}-\Pi_{p,h_{\ell_2}}$ is the $H^r$-orthogonal complementary projection, it is a linear operator and it verifies:
\[
 \|\operatorname{I}-\Pi_{p,h_{\ell_2}}\|_{\mathscr{L}(H^q(\hat{\Omega}))}=1.
 \]
 Then, for all $0\leq s \leq r$,
\[
\left\|\frac{\partial^s}{\partial y^s} \Bigl[ \Bigl(\operatorname{I}-\Pi_{p,h_{\ell_2}} \Bigr)u(x,y)\Bigr] \right\|_{H^r(\hat{\Omega})} \leq \left\| \frac{\partial^s}{\partial y^s} u(x,y)\right\|_{H^r(\hat{\Omega})}, \quad \forall  x,y\in (0,1).
\]
As a result,
\[
\frac{\partial^s}{\partial y^s} \Bigl[ \Bigl(\operatorname{I}-\Pi_{p,h_{\ell_2}} \Bigr)u(\cdot,y)\Bigr]\in H^r(0,1)\subset H^m(0,1), \quad \forall y\in (0,1),
\]
because $u\in H^q(\hat{\Omega})$ and $0\leq m\leq r \leq q$. Because of linearity we have:
\[
\frac{\partial^s}{\partial y^s}\Bigl[(\operatorname{I}-\Pi_{p,h_{\bm{\ell}}})^{\{1,2\}}u(\cdot,y)\Bigr] = (\operatorname{I}-\Pi_{p,h_{\ell_1}})\Bigl[\frac{\partial^s}{\partial y^s}(\operatorname{I}-\Pi_{p,h_{\ell_2}})u(\cdot,y)\Bigl].
\]
According to \cref{lem:2}, 
\begin{align}
\label{lem:5:eq:3}
\left|\frac{\partial^s}{\partial y^s}(\operatorname{I}-\Pi_{p,h_{\bm{\ell}}})^{\{1,2\}}u(\cdot,y) \right|_{H^m(0,1)}\leq \left(\revB{\frac{1}{\pi}}h_{\ell_1}\right)^{q-m} \left| \frac{\partial^s}{\partial y^s}(\operatorname{I}-\Pi_{p,h_{\ell_2}})u(\cdot,y)\right|_{H^q(0,1)}.
\end{align}
By squaring and taking the integral in $y$ on both sides, we obtain: 
\begin{align}
\label{lem:5:eq:4}
\left\|\frac{\partial^m}{\partial x^m}\frac{\partial^s}{\partial y^s}(\operatorname{I}-\Pi_{p,h_{\bm{\ell}}})^{\{1,2\}}u\right\|_{L^2(\hat{\Omega})}^2\leq  \left(\revB{\frac{1}{\pi}}h_{\ell_1}\right)^{2(q-m)} \int_0^1 \left|(\operatorname{I}-\Pi_{p,h_{\ell_2}})\frac{\partial^q}{\partial x^q}u(x,\cdot)\right|_{H^s(0,1)}^2dx.
\end{align}
Since $\displaystyle \frac{\partial^q}{\partial x^q}u(x,\cdot) \in H^s(0,1)$ for all $x\in(0,1)$, we can apply again \cref{lem:2} to obtain, for all $0\leq s,m\leq r$,
\begin{align}
\left\|\frac{\partial^m}{\partial x^m}\frac{\partial^s}{\partial y^s}(\operatorname{I}-\Pi_{p,h_{\bm{\ell}}})^{\{1,2\}}u\right\|_{L^2(\hat{\Omega})}^2&\leq  \left(\revB{\frac{1}{\pi}}h_{\ell_1}\right)^{2(q-m)}  \left(\revB{\frac{1}{\pi}}h_{\ell_2}\right)^{2(q-s)}  \int_0^1 \left|\frac{\partial^q}{\partial x^q}u(x,\cdot)\right|_{H^q(0,1)}^2dx \nonumber  \\
 &\leq  \left(\revB{\frac{1}{\pi^2}}h_{\ell_1}h_{\ell_2}\right)^{2(q-r)}\|u\|^2_{H^q_{\mathrm{mix}}(\hat{\Omega})}. \label{lem:5:eq:1}
\end{align}
The second inequality is obtained by using the relations:
\[
\revB{\frac{1}{\pi}}h_{\ell_i}<1  \quad \forall i=1,\ldots,d, \quad \text{and} ~
\displaystyle \left\|\frac{\partial^q}{\partial x^q}\frac{\partial^q}{\partial y^q} u\right\|_{L^2(\hat{\Omega})}\leq \|u\|_{H^q_{\mathrm{mix}}(\hat{\Omega})}.
\]
Each term in the $H^r_{\mathrm{mix}}(\hat{\Omega})$ norm can be bounded using \cref{lem:5:eq:1}, with $0\leq s,m\leq r$. Since there are $(r+1)^d$ terms in the seminorm, we obtain the result for $k=2$.

Concerning the terms corresponding to projections along a single direction, i.e., $k=1$, since:
\[
\displaystyle \frac{\partial^s}{\partial y^s}u(\cdot,y)\in H^m(0,1), ~ \forall y\in (0,1) \quad \text{and} ~ \displaystyle \frac{\partial^s}{\partial x^s}u(x,\cdot)\in H^m(0,1)), ~ \forall x\in (0,1),
\]
we can apply \cref{lem:2} to these two functions and use the same steps as in \cref{lem:5:eq:3}-\cref{lem:5:eq:1} to obtain: 
\begin{align}
\label{lem:5:eq:2}
\left\|\frac{\partial^m}{\partial x^m}\frac{\partial^s}{\partial y^s}(\operatorname{I}-\Pi_{p,h_{\ell_j}})u\right\|_{L^2(\hat{\Omega})}\leq \left(\revB{\frac{1}{\pi}}h_{\ell_j}\right)^{q-r} \|u\|_{H^q_{\mathrm{mix}}(\hat{\Omega})}, \quad \text{for}~j=1,2.
\end{align}
We conclude by bounding all of the $(r+1)^d$ terms in the norm using \cref{lem:5:eq:2}.
\end{proof}

\begin{proof}[Proof of \cref{lem:3}]
\cref{lem:3:eq:0} can be recast as:
\begin{align*}
\sum_{l=0}^{d-1}&(-1)^l\binom{d-1}{l}\binom{|\log h|+d-2-l-\ell}{d-1-k}\\
&=\sum_{l=0}^{d-1}\binom{d-1}{l}\frac{(-1)^l}{(d-1-k)!}\prod_{j=1}^{d-1-k} (|\log h|-l-\ell-1+k+j).\\ 
\end{align*}
By expanding the products in the right side we obtain a sum of terms with powers $(-l)^{i}$ like the left side of \cref{lem:3:eq:1} in \cref{lem:1}, with $ 0\leq i\leq d-1$. As a result, the expression vanishes for $k>0$. If $k=0$,  only one term remains, the one with $(-l)^{d-1}$. Then by applying $d-2$ times the same calculations than in \cref{lem:3:eq:6}-\cref{lem:3:eq:4}, we obtain:
\begin{align}
\underset{l=0}{\overset{d-1}{\sum}}(-1)^l\binom{d-1}{l}(-l)^{d-1}&=(-1)^{d-2}(-1)^{d-1}\frac{(d-1)!}{(d-1)!}\sum_{l=0}^{1}(-1)^l\binom{1}{l}l=1. 
\end{align}
\end{proof}

\begin{proof}[Proof of \cref{lem:4}]
The sum can be recast into:
\begin{align}
\sum_{\bm{\ell}\in \mathscr{L}_h}c_{\bm{\ell}}&=\sum_{l=0}^{d-1}(-1)^l\binom{d-1}{l}\sum_{\bm{\ell}\in \mathscr{L}_{h,l}}1 \nonumber \\
&=\sum_{l=0}^{d-1}(-1)^l\binom{d-1}{l}\binom{|\log h|+d-2-l}{d-1} \label{lem:4:eq:2}\\
&=1 \label{lem:4:eq:3}.
\end{align}
We obtain \cref{lem:4:eq:2} by counting the number of elements in $ \mathscr{L}_{h,l}$; and \cref{lem:4:eq:3} by using \cref{lem:3}.
\end{proof}

\begin{proof}[Proof of \cref{thm:2}]
Let $u_{p,h}\in S_{p,h}^{(C)}(\hat{\Omega})$, then by definition:
\[
u_{p,h} = \sum_{\bm{\ell}\in \mathscr{L}_h}c_{\bm{\ell}} u_{p,h_{\bm{\ell}}}, \quad \text{where} ~~u_{p,h_{\bm{\ell}}}\in S_{p,h_{\bm{\ell}}}(\hat{\Omega}).
\]
Using the definition of the combination coefficients and the direct sum decomposition in \cref{eq:4}, we have:
\begin{align}
u_{p,h} &=\sum_{l=0}(-1)^l\binom{d-1}{l}\sum_{\bm{\ell} \in \mathscr{L}_{h,l}}\sum_{\bm{k}\leq \bm{\ell}} w_{p,h_{\bm{k}}} \label{thm:2:eq:1}\\
&= \sum_{\bm{k}} w_{p,h_{\bm{k}}} N(\bm{k},h,d,p). \label{thm:2:eq:2}
\end{align}
We have rearranged the sums in \cref{thm:2:eq:1} and introduced in \cref{thm:2:eq:2} the following quantity:
\[
N(\bm{k},h,d,p) = \left\{ \begin{array}{ll}
\displaystyle \sum_{l=0}(-1)^l\binom{d-1}{l}\sum_{\bm{\ell} \in \mathscr{L}_{h,l}} 1 \quad \text{if}~ |\bm{k}|_1\leq |\log h|+d-1,\\
0 \quad \text{if}~ |\bm{k}|_1> |\log h|+d-1. \end{array}\right. 
\]
By counting the number of elements in $\mathscr{L}_l$, and using \cref{lem:3}, we have: 
\[
N(\bm{k},h,d,p) = \left\{ \begin{array}{ll}
1 \quad \text{if}~ |\bm{k}|_1\leq |\log h|+d-1,\\
0 \quad \text{if}~ |\bm{k}|_1> |\log h|+d-1. \end{array}\right. 
\]
As a result, $u_{p,h}\in S_{p,h}^{(1)}(\hat{\Omega})$ and $S_{p,h}^{\mathscr{L}}(\hat{\Omega})\subset S_{p,h}^{(1)}(\hat{\Omega})$. By performing the same substitutions in the sums in reverse, we can we also show that $S_{p,h}^{(1)}(\hat{\Omega})\subset S_{p,h}^{(C)}(\hat{\Omega})$, which conclude the proof.
\end{proof}

\begin{proof}[Proof of \cref{lem:8}]
We recast the expression of the left side of \cref{lem:8:eq:1} by introducing the combination coefficient expression:
\begin{align}
\label{lem:8:eq:2}
\sum_{\bm{\ell}\in \mathscr{L}_h}c_{\bm{\ell}}\sum_{k=1}^d  \sum_{\substack{J\subset \{1,\ldots,d\}\\ |J|=k}} a_{J,\bm{\ell}} = \sum_{k=1}^d  \sum_{\substack{J\subset \{1,\ldots,d\}\\ |J|=k}}\sum_{l=0}^{d-1}(-1)^l\binom{d-1}{l} \sum_{\bm{\ell}\in \mathscr{L}_{h,l}}a_{J,\bm{\ell}}.
\end{align}
Note that, in the above, the admissible levels verify a condition on the $\ell^1$ norm of all components and not only the $J$ directions.

If $k=d$, we directly obtain the second term of \cref{lem:8:eq:1}, for which there is no cancellation.  If $ k \leq d-1$,  cancellations between the different admissible levels occur in the combination.
The first step is to rearrange the innermost sums of the right side of \cref{lem:8:eq:2}, so that we can extract all the terms that will be cancelled. Specifically, we have: 
\begin{align}
 &\sum_{l=0}^{d-1}(-1)^l\binom{d-1}{l}  \sum_{\bm{\ell}\in \mathscr{L}_{h,l}}a_{J,\bm{\ell}} \nonumber \\
 &=\sum_{l=0}^{d-1}(-1)^l\binom{d-1}{l}\sum_{\bm{\ell}\in \mathscr{L}^J_{h,l,k}}\binom{|\log h|+d-2-l-(\ell_{J_1}+\ldots+\ell_{J_k})}{d-1-k}\,a_{J,\bm{\ell}}\label{lem:8:eq:4}\\
 &= \sum_{\bm{\ell}\in \tilde{\mathscr{L}}^J_{h,d-1,k}} a_{J,\bm{\ell}}\sum_{\kappa=0}^{d-1} (-1)^\kappa\binom{d-1}{\kappa} \binom{|\log h|+d-2-\kappa-(\ell_{J_1}+\ldots+\ell_{J_k})}{d-1-k}, \nonumber \\
 &+\sum_{l=0}^{d-2}  \sum_{\bm{\ell}\in \mathscr{L}^J_{h,l,k}}a_{J,\bm{\ell}}\sum_{\kappa=0}^{l} (-1)^\kappa\binom{d-1}{\kappa} \binom{|\log h|+d-2-\kappa-(\ell_{J_1}+\ldots+\ell_{J_k})}{d-1-k}, \label{lem:8:eq:3}
\end{align}
We have introduced the set of levels
\[
\tilde{\mathscr{L}}^J_{h,l,k} := \Bigl\{\bm{\ell}\in \mathbb{N}^d ~~| ~~ \sum_{i\in J} |\ell_{i}|\leq  |\log h|+k-1-l, ~~ \ell_{J_j}\geq 1, ~~ \text{for} ~ j=1,\ldots,k  \Bigr\}.
\]
In \cref{lem:8:eq:4}, we have rearranged the sum on the levels, by summing in the $J$ directions and letting the remaining directions constraint by the $l^1$-norm equality. We have rearranged again the sum on the levels into two terms: whether if the level appears in each term $0\leq l \leq d-1$ of the sum in \cref{lem:8:eq:4}, which correspond to the first term in \cref{lem:8:eq:3}; or not, which correspond to the second one. 

According to \cref{lem:3}, the first term in \cref{lem:8:eq:3} is equal to zero. This correpsonds to coarse contributions that are cancelled in the combination technique. Substituting $(\ell_{J_1}+\ldots+\ell_{J_k})$ by its expression in \cref{lem:8:eq:3}, we obtain the result for $1\leq k \leq d-1$.
\end{proof}

\begin{proof}[Proof of \cref{lem:6}]
The local interpolation error can be written as
\begin{align}
(\operatorname{I}-\Pi_{p,h}^{(C)})u &= \Big(\operatorname{I}-\sum_{\bm{\ell}\in \mathscr{L}_h}c_{\bm{\ell}}\Pi_{p,h_{\bm{\ell}}}\Bigr)u \nonumber\\
&= \sum_{\bm{\ell}\in \mathscr{L}_h}c_{\bm{\ell}}(\operatorname{I}-\Pi_{p,h_{\bm{\ell}}})u \label{lem:6:eq:2} \\
&=  \sum_{\bm{\ell}\in \mathscr{L}_h}c_{\bm{\ell}} \sum_{k=1}^d   \sum_{\substack{J\subset \{1,\ldots,d\}\\ |J|=k}}  (-1)^k (\operatorname{I}-\Pi_{p,h_{\bm{\ell}}})^{J}u \label{lem:6:eq:3} \\
&= \sum_{k=1}^{d-1}  \sum_{l=1}^{d-2}  C_4(d,k,l) (-1)^{k-1}  \sum_{\substack{J\subset \{1,\ldots,d\}\\ |J|=k}} \sum_{\bm{\ell}\in \mathscr{L}_{h,l,k}^J} (\operatorname{I}-\Pi_{p,h_{\bm{\ell}}})^{J}u  \nonumber \\ 
&+\sum_{l=0}^{d-1} C_5(d,l) \sum_{\bm{\ell}\in \mathscr{L}_{h,l}} (\operatorname{I}-\Pi_{p,h_{\bm{\ell}}})^{\{1,\ldots,d\}}u. \label{lem:6:eq:4}
\end{align}
In the above, \cref{lem:4} has been used in \cref{lem:6:eq:2}; \cref{lem:7} has been used in \cref{lem:6:eq:3}; and \cref{lem:8} has been used in \cref{lem:6:eq:4}. Then, using the triangular inequality and applying \cref{lem:5} to each of the terms, we obtain:
\begin{align}
\left\|(\operatorname{I}-\Pi_{p,h}^{(C)})\,u  \right\|_{H^r_{\mathrm{mix}}(\hat{\Omega})} &\leq \sum_{k=1}^{d-1}  \sum_{l=0}^{d-2}  C_6(d,r,k,l)   \sum_{\substack{J\subset \{1,\ldots,d\}\\ |J|=k}}\sum_{\bm{\ell}\in \mathscr{L}_{h,l,k}^J}  \Big( \prod_{i\in J}h_{\ell_i}^{q-r}\Bigr) \|u\|_{H^q_{\mathrm{mix}}(\hat{\Omega})} \nonumber \\
&+ \sum_{l=0}^{d-1} C_7(d,r,l)\sum_{\bm{\ell}\in \mathscr{L}_{h,l}}  \Big( \prod_{i\in \{1,\ldots,d\}}h_{\ell_i}^{q-r}\Bigr) \|u\|_{H^q_{\mathrm{mix}}(\hat{\Omega})}, 
\label{lem:6:eq:5}
\end{align} 
where the constants are
\[
C_6(d,r,k,l) = C_3(d,k,l) C_2(d,k,q,r) , \quad C_7(d,r,l) = C_4(d,l) C_2(d,d,q,r).
\]
The number of elements in each sum is given by
\begin{align*}
&\sum_{\substack{J\subset \{1,\ldots,d\}\\ |J|=k}} 1 = \binom{d}{k}, \quad \sum_{\bm{\ell}\in \mathscr{L}_{h,l,k}^J} 1 = \binom{|\log h|+k-2-l}{k-1}, \quad \sum_{\bm{\ell}\in \mathscr{L}_{h,l}} 1=\binom{|\log h|+d-2-l}{d-1}.
\end{align*}
Since the last two sums are maximal for $l=0$, \cref{lem:6:eq:5} can be bounded by
\begin{align*}
\left\|\left(I-\Pi_{p,h}^{(C)}\right)u  \right\|_{H^r_{\mathrm{mix}}(\hat{\Omega})} &\leq  \sum_{k=1}^{d-1}C_8(d,q,r,k)\binom{|\log h|+k-2}{k-1}h^{q-r}\|u\|_{H^q_{\mathrm{mix}}}\\
&+C_9(d,q,r)\binom{|\log h|+d-2}{d-1}h^{q-r}\|u\|_{H^q_{\mathrm{mix}}},
\end{align*}
where
\[
C_8(d,q,r,k) =\binom{d}{k} \sum_{l=0}^{d-2}2^{-(q-r)(k-1-l)}C_5(d,r,k,l), \quad C_9(d,q,r)= \sum_{l=0}^{d-2}2^{-(q-r)(d-1-l)}C_6(d,r,l).
\]
Eventually, we conclude with the bound
\[
 \binom{|\log h|+d-2}{d-1}\leq  \frac{(d-1)^{d-1}}{(d-1)! } |\log h|^{d-1},
\]
so that the constant does not depend on $p$.
\end{proof}

\begin{proof}[Proof of \cref{lem:11}]
Like in the proof of \cref{lem:5}, we demonstrate the result for $d=2$, as the extension to higher dimensions uses the same arguments. Let $0\leq s,m \leq q$ be two integers. The hierarchical increment verifies:
\[
w_{p,h_{\bm{\ell}}}\in W_{p,h_{\bm{\ell}}}(\hat{\Omega})\subset H^p(\hat{\Omega}) \subset H^q(\hat{\Omega}), %H^q_{\mathrm{mix}},
\] 
%so that we assume $w_{p,h_{\bm{\ell}}}\in \mathcal{C}^{\infty}(\hat{\Omega})$ and conclude using a density argument, since $\mathcal{C}^{\infty}(\hat{\Omega})$ is dense in $H^p(\hat{\Omega})$.
%$H^q_{\mathrm{mix}}(\hat{\Omega})$. 
so that we have:
\begin{align*}
&\frac{\partial^s}{\partial y^s}w_{p,h_{\bm{\ell}}}(\cdot,y)\in H^m(0,1), \quad \forall y\in(0,1), \\
&w_{p,h_{\bm{\ell}}}(x,\cdot)\in H^s(0,1), \quad \forall x\in(0,1).
\end{align*}
Therefore, we can demonstrate the bounds
\begin{align}
\Bigl\|\frac{\partial^{m}}{\partial x^{m}}\frac{\partial^{s}}{\partial y^{s}}w_{p,h_{\bm{\ell}}}\Bigr\|_{L^2(\hat{\Omega})}^2 &= \int_0^1 \Bigl| \frac{\partial^{s}}{\partial y^{s}}w_{p,h_{\bm{\ell}}}(\cdot,y)\Bigr|^2_{H^m(0,1)}dy \nonumber \\
&\leq (2\sqrt{3})^{2m} h_{\ell_1}^{-2m}\int_0^1\int_0^1 \Bigl( \frac{\partial^{s}}{\partial y^{s}}w_{p,h_{\bm{\ell}}}(x,y)\Bigr)^2dxdy \label{lem:11:eq:2}  \\
&\leq (2\sqrt{3})^{2m} h_{\ell_1}^{-2m}\int_0^1 \Bigl|w_{p,h_{\bm{\ell}}}(x,\cdot)\Bigr|^2_{H^s(0,1)}dx \nonumber \\
&\leq (2\sqrt{3})^{2(m+s)} h_{\ell_1}^{-2m} h_{\ell_2}^{-2s} \|w_{p,h_{\bm{\ell}}}\|_{L^2(\hat{\Omega})}^2\label{lem:11:eq:4}.
\end{align}
We have used \cref{lem:12} twice: in \cref{lem:11:eq:2} and in \cref{lem:11:eq:4}.
\end{proof}

\begin{proof}[Proof of \cref{lem:10}]
%Let $0\leq s_1,\ldots,s_d \leq q \leq p+1$ be integers denoting the derivatives in each direction. 
The following inequalities are true:
\begin{align}
\Bigl\|\frac{\partial^{s_1}}{\partial x_1^{s_1}}\ldots\frac{\partial^{s_d}}{\partial x_d^{s_d}}u_{p,h}^{(1)}\Bigr\|_{L^2(\hat{\Omega})}& \leq \sum_{\bm{\ell}\in \mathscr{H}_h} \Bigl\|\frac{\partial^{s_1}}{\partial x_1^{s_1}}\ldots\frac{\partial^{s_d}}{\partial x_d^{s_d}}w_{p,h_{\bm{\ell}}}\Bigr\|_{L^2(\hat{\Omega})}  \label{lem:10:eq:1}\\
&\leq \sum_{\bm{\ell}\in \mathscr{H}_h}(2\sqrt{3})^{|\bm{s}|_1}  \Bigl(\prod_{i=1}^d h_{\ell_i}^{-s_i}\Bigr)\|w_{p,h_{\bm{\ell}}}\|_{L^2(\hat{\Omega})}\label{lem:10:eq:2} \\
&\leq (2\sqrt{3})^{dq} \Bigl( \sum_{\bm{\ell}\in \mathscr{H}} \prod_{i=1}^d h_{\ell_i}^{-2s_i}\Bigr)^{1/2} \|u_{p,h}^{(1)}\|_{L^2(\hat{\Omega})}. \label{lem:10:eq:3}
\end{align}
We have used the triangular inequality in \cref{lem:10:eq:1}; applied \cref{lem:11} in \cref{lem:10:eq:2}; and used Cauchy-Schwarz inequality in \cref{lem:10:eq:3}. We then have the relation
\begin{align}
\label{lem:10:eq:4}
\sum_{\bm{\ell}\in \mathscr{H}_h} \prod_{i=1}^d h_{\ell_i}^{-2s_i} \leq h^{-2q} 2^{2(d-1)}\Bigr(\sum_{\bm{\ell}\in \mathscr{H}_h} 1\Bigl),
\end{align}
and we bound the last term of the product by counting the number of elements in $ \mathscr{H}$:
\begin{align}
\label{lem:10:eq:5}
\sum_{\bm{\ell}\in \mathscr{H}_h} 1=\sum_{\substack{\bm{k}\in \mathbb{N}^d\\ |\bm{k}|_1\leq |\log h|}} 1\leq \dbinom{|\log h|+d}{d}.
\end{align}
Gathering \cref{lem:10:eq:4} and \cref{lem:10:eq:5} together, and since $\lambda_p\geq 1$, we obtain:
\begin{align}
\label{lem:10:eq:6}
\sum_{\bm{\ell}\in \mathscr{H}_h} \prod_{i=1}^d h_{\ell_i}^{-2s_i}  \leq  h^{-2q} 2^{2(d-1)}\dbinom{|\log h|+d-1}{d} \leq h^{-2q} \log |h|^{d} \frac{2^{2(d-1)} 2^d}{d!}.
\end{align}
Eventually, we substitute the side of \cref{lem:10:eq:6} relation into \cref{lem:10:eq:3} and bound all the $(q+1)^d$ in the $H^q_{\mathrm{mix}}$ full norm to get:
\[
\Bigl\|u_{p,h}^{(1)}\Bigr\|_{H^q_{\mathrm{mix}}(\hat{\Omega})}\leq (q+1)^{d/2} (2\sqrt{3})^{dq}\frac{2^{(d-1)} 2^{d/2}}{d! } h^{-q} \log |h|^{d/2}.
\]
\end{proof}

\begin{proof}[Proof of \cref{lem:9}]
%\cref{lem:9:eq:1} is a result from~\cite{bazilev11}, Lemma 3.5.
The proof uses similar arguments than \cite{ciarlet72}, Lemma 3, and is based on the multivariate Fa\'a di Bruno decomposition formula. Let $\bm{\alpha}$ be such that $|\bm{\alpha}|_\infty=q$. Using the Fa\'a di Bruno formula, the triangular inequality, the relation $\|f \cdot(\prod_j g_j)\|_{L^2}\leq \|f\|_{L^2}\prod_j \|g_j\|_{L^\infty} $, and putting together the
factors which correspond to identical values of indices of derivatives of $\bm{F}_{h_0}$, we obtain: 
\begin{align}
\|D^{\bm{\alpha}} &(u\circ \bm{F}_{h_0})\|_{L^2(\hat{\Omega})} \nonumber\\
&\leq C(q) \sum_{m=0}^{dq}\|D^{m} u(\bm{F}_{h_0})\|_{L^2(\hat{\Omega})} \sum_{\bm{i}\in I(m,q,\bm{\alpha})} \prod_{j=1}^q  \|\bm{\nabla}^{j} \bm{F}_{h_0}\|^{i_j}_{L^\infty(\hat{\Omega})} \nonumber\\
&\leq C(q) \|\det \bm{\nabla} \bm{F}_{h_0}^{-1}\|^{1/2}_{L^\infty(\Omega)} \sum_{m=0}^{dq}\|\bm{\nabla}^{m} u\|_{L^2(\Omega)} \sum_{\bm{i}\in I(m,q,\bm{\alpha})} \prod_{j=1}^q  \|\bm{\nabla}^{j} \bm{F}_{h_0}\|^{i_j}_{L^\infty(\hat{\Omega})}.\label{lem:9:eq:4} 
\end{align}
In the above, we have introduced the following index set: 
\[
 I(m,q,\bm{\alpha}) = \{\bm{i}\in \mathbb{N}^q ~~ | ~~ |\bm{i}|_1=m , ~~ 1\cdot i_1 + 2\cdot i_2 +\ldots+q \cdot i_q = |\bm{\alpha}|_1\}.
\]
\cref{lem:9:eq:4} is obtained with a change of variable. Then, using the H\"older inequality, and gathering all the terms in a constant $C_{12}$, we get:
\[
\|u\circ \bm{F}_{h_0}\|_{H^q_{\mathrm{mix}}(\hat{\Omega})} \leq \sum_{m=0}^{dq} C_{12}(m,q,\|\bm{\nabla F}_{h_0}\|_{W^{q,\infty}(\hat{\Omega})}, \|\det \bm{\nabla} \bm{F}_{h_0}^{-1}\|_{L^\infty(\Omega)}) |u|_{H^m(\Omega)}.
\]
We obtain \cref{lem:9:eq:1} by applying the same arguments to $\|D^{\bm{\alpha}} u\|_{L^2(\hat{\Omega})}=\|D^{\bm{\alpha}} (u\circ \bm{F}_{h_0}\circ \bm{F}_{h_0}^{-1})\|_{L^2(\hat{\Omega})}$, with $|\bm{\alpha}|_1=q$.
\end{proof}

\begin{proof}[Proof of \cref{thm:1}]
For simplicty of notations, we omit the parameters dependence of the constants. The interpolation error in the $H^r(\Omega)$ seminorm is bounded by:
\begin{align}
&\left|(\operatorname{I}-\Pi_{p,h}^{(C)})\,u\right|_{H^r(\Omega)} = \left|\Bigl[(\operatorname{I}-\Pi_{p,h}^{(C)})\,(u\circ \bm{F})\Bigr]\circ\bm{F}^{-1}\right|_{H^r(\Omega)} \nonumber \\
&\leq  \Bigl(\max_{0\leq m \leq dq} C_{12}\Bigr)  \sum_{m=0}^r \left|(\operatorname{I}-\Pi_{p,h}^{(C)})\,(u\circ \bm{F}_{h_0})\right|_{H^m(\Omega)}\label{thm:1:eq:2}  \\
&\leq  \Bigl(\max_{0\leq m \leq dq} C_{12}\Bigr) \sqrt{r}  \left\|(\operatorname{I}-\Pi_{p,h}^{(C)})\,(u\circ \bm{F}_{h_0})\right\|_{H^r(\Omega)} \label{thm:1:eq:3} \\
&\leq  \Bigl(\max_{0\leq m \leq dq} C_{12}\Bigr)  \sqrt{r}  \left\|(\operatorname{I}-\Pi_{p,h}^{(C)})\,(u\circ \bm{F}_{h_0})\right\|_{H^r_{\mathrm{mix}}(\Omega)}\label{thm:1:eq:4}  \\
&\leq  \Bigl(\max_{0\leq m \leq dq} C_{12}\Bigr)  C_{10} \sqrt{r}\, h^{q-r} |\log h|^{d-1} \left\|u\circ \bm{F}_{h_0}\right\|_{H^q_{\mathrm{mix}}(\Omega)} \label{thm:1:eq:5} \\
&\leq  \Bigl(\max_{0\leq m \leq dq} C_{12}\Bigr) C_{10} \sqrt{r}\, h^{q-r} |\log h|^{d-1} \sum_{m=0}^{dq} C_{13}\left|u\right|_{H^{m}(\Omega)}\label{thm:1:eq:6}  \\
&\leq \Bigl(\max_{0\leq m \leq dq} C_{12}\Bigr) C_{10} \Bigl(\max_{0\leq m \leq dq} C_{13}\Bigr)\sqrt{dqr}\, h^{q-r} |\log h|^{d-1}\left\|u\right\|_{H^{dq}(\Omega)}\label{thm:1:eq:7}. 
\end{align} 
In the above, \cref{lem:9:eq:1} of \cref{lem:9} has been used to obtain \cref{thm:1:eq:2}; Cauchy-Schwarz inequality has been used in \cref{thm:1:eq:3}; and the norm inequalities \cref{eq:13} has been used in \cref{thm:1:eq:4}. Then, we used \cref{lem:6} to obtain \cref{thm:1:eq:5} and we apply  \cref{lem:9:eq:2} of \cref{lem:9} to get \cref{thm:1:eq:6}. Eventually applying again the Cauchy-Schwarz inequality we obtain the result of \cref{thm:1:eq:7}. 
\end{proof}

\begin{proof}[Proof of \cref{thm:3}]
For simplicty of notations, we omit the parameters dependence of the constants. Let $u_{p,h}^{(1)}\in \tilde{S}_{p,h}^{q,(1)}(\Omega)$, then its $H^q(\Omega)$ seminorm error is bounded by:
\begin{align}
|u_{p,h}^{(1)}\circ |_{H^{q}(\Omega)} &= |(u_{p,h}^{(1)}\circ \bm{F}_{h_0})\circ \bm{F}^{-1}|_{H^{q}(\Omega)} \nonumber \\
&\leq \sum_{m=0}^q C_{12} |u_{p,h}^{(1)}\circ \bm{F}_{h_0}|_{H^{m}(\hat{\Omega})} \label{thm:3:eq:2} \\
&\leq \sqrt{q} C_{12} \|u_{p,h}^{(1)}\circ \bm{F}_{h_0}\|_{H^{q}(\hat{\Omega})}  \label{thm:3:eq:3}\\
&\leq \sqrt{q} C_{12} \|u_{p,h}^{(1)}\circ \bm{F}_{h_0}\|_{H^{q}_{\mathrm{mix}}(\hat{\Omega})}  \label{thm:3:eq:4}\\
& \leq \sqrt{q} C_{12}C_{11} h^{-q} |\log h|^{d/2} \|u_{p,h}^{(1)}\circ \bm{F}\|_{L^2(\hat{\Omega})}  \label{thm:3:eq:5} \\
& \leq \sqrt{q} C_{12}C_{11}C_{13} h^{-q} |\log h|^{d/2} \|u_{p,h}^{(1)}\|_{L^2(\Omega)}  \label{thm:3:eq:6} 
\end{align} 
In the above, \cref{lem:9:eq:1} of \cref{lem:9} has been used to obtain \cref{thm:3:eq:2}; Cauchy-Schwarz inequality has been used in \cref{thm:3:eq:3}; and the norm inequalities \cref{eq:13} has been used in \cref{thm:3:eq:4}. Then, we used \cref{lem:10} to obtain \cref{thm:3:eq:5} and eventually we apply  \cref{lem:9:eq:2} of \cref{lem:9} with $q$=0 to get \cref{thm:3:eq:6}.
\end{proof}

\begin{proof}[Proof of \cref{prop:1}]
The dimension of the physical domain approximation space coincides with that of the parameter domain space. To prove the result, we choose the sparse-grid combination technique formulation and count the number of basis functions involved in the combination:
\begin{align}
\label{prop:1:eq:1}
\dim S_{p,h}^{(C)}(\hat{\Omega})  \leq \sum_{l=0}^{d-1}\Bigl|(-1)^l\Bigr|\left|\binom{d-1}{l}\right|\sum_{\substack{|\bm{\ell}|_1=|\log h|+d-1-l, \\ \ell_j \geq 1}}\dim S_{p,h_{\bm{\ell}}}(\hat{\Omega}).
\end{align}
Since all of the basis functions of $S_{p,h_{\bm{\ell}}}(\hat{\Omega})$, for a given level $\bm{\ell}\in \mathbb{N}^d$, are linearly independent we have:
\begin{align}
\label{prop:1:eq:2}
\dim S_{p,h_{\bm{\ell}}}(\hat{\Omega})  = \prod_{l=1}^d (h_{\ell_l}^{-1}+p).
\end{align}
Substituting \cref{prop:1:eq:2} into \cref{prop:1:eq:1}, we obtain:
\[
\dim S_{p,h}^{(C)}(\hat{\Omega})\leq h^{-1} \sum_{l=0}^{d-1}\binom{d-1}{l}\binom{|\log h|+d-2-l}{d-1}(1+p)^d2^{-(d-1-l)}.
\]
We conclude with the estimation
\[
\binom{|\log h|+d-2-l}{d-1} \lesssim |\log h|^{d-1}.
\]
\end{proof}

\end{document}